\documentclass[12pt]{amsart}
\usepackage{amsmath}
\usepackage{amssymb}
\usepackage{tabularx}
\usepackage{enumerate}
\usepackage[dvipdfm]{graphicx}
\usepackage{texdraw}
\usepackage{color}
\usepackage{lineno}
\usepackage{tikz}
\usetikzlibrary {shapes.geometric}

\usepackage{mathrsfs}
\usepackage{amsfonts,amssymb,amsmath}
\usepackage{epsfig}
\usepackage{subfigure}
\usepackage{hyperref}
\usepackage{float}
\makeatletter
\@addtoreset{equation}{section}

\makeatother
\newtheorem{theorem}{Theorem}[section]
\newtheorem{lemma}{Lemma}[section]

\newtheorem{proposition}{Proposition}[section]
\newtheorem{remark}{Remark}[section]

\begin{document}
\title[Fisher-KPP waves and the minimal speed on hexagonal lattice]
{Fisher-KPP waves and the minimal speed on hexagonal lattice$^\S$}
\thanks{$\S$ Jian Fang and Jian Wang are supported in part by NSF of China (No. 12171119). Yifei Li is supported in part by NSF of China (No. 12301624). Yijun Lou is partially supported by The Hong Kong Research Grants Council (No. 15304821).}
\author[J. Fang, Y.  Li, Y.  Lou and J. Wang]{Jian Fang$^\dag$, Yifei Li$^\ddag$, Yijun Lou$^\sharp$ and Jian Wang$^\ddag$}
\thanks{$\dag$  Institute for Advanced Studies in Mathematics and School of Mathematics, Harbin Institute of Technology, Harbin 150001, China.}
\thanks{$\ddag$ School of Mathematics, Harbin Institute of Technology, Harbin 150001, China.}
\thanks{$\sharp$ Department of Applied Mathematics, Hong Kong Polytechnic University, Hong Kong, China.}
\thanks{E-mails: {\sf jfang@hit.edu.cn} (J. Fang), {\sf yifeili@hit.edu.cn} (Y.  Li), {\sf yijun.lou@polyu.edu.hk} (Y.  Lou), {\sf jwang815@163.com} (J. Wang).}
\date{\today}

\begin{abstract}
The hexagonal structure is ubiquitous in nature. The propagation phenomena occurring in a media with a hexagonal structure remain to be explored. One way of exploring this question is to formulate lattice dynamical systems and analyze the propagation dynamics. In this paper, we propose a lattice differential equation model featuring a discrete diffusion operator with the hexagonal structure, and a monostable nonlinear term known as the Fisher-KPP mechanism in modeling population growth. A rigorous analysis is conducted on the traveling waves, thoroughly establishing the existence and uniqueness (up to translation) of the traveling waves. Moreover, the periodicity and monotonicity of the minimal wave speed concerning an angle are demonstrated, which is different from the existing results of the minimal wave speed in $\mathbb{R}^2$ and $\mathbb{Z}^2$. Numerical results validate  theoretical analysis and further suggest that the minimal wave speed is also the spreading speed of solutions with compactly supported initial values.
\end{abstract}

\subjclass[2010]{34A33, 35C07}
\keywords{hexagonal lattice, Fisher-KPP equation, traveling waves, minimal wave speed, existence and uniqueness}
\maketitle

%%%%%%%%%%%%%%%%%%%%%%%%%%%%%%%%%%%%%%%%%%%%%%%%%%%%%%%%%%%%%%%%%%%%%%%%%%%%%%%%%%%%%%%%%%%%%%%%%%%%%%%%%%%%%%%%%%%%%

\section{Introduction}\label{sec1}

The study of spatial distribution of species is an important research area in ecology and conservation biology, with significant implications for understanding biodiversity, ecosystem dynamics, and the impacts of various interacting factors. Mathematical modeling has shown its power in this field by providing insights into the mechanisms determining the species patterns across landscapes. Integrating mathematical frameworks with ecological data enable researchers to simulate and predict patterns of range expansion, offering valuable guidance for conservation strategies and management practices.

Spatial models typically integrate various biological and environmental factors, such as reproduction and death rates, dispersal mechanisms, habitat suitability, and interspecies interactions. These models can range from deterministic equations to complex stochastic models. The  most widely used deterministic frameworks include reaction-diffusion equations \cite{CC, LL}, which describe the change in population densities over space and time,  integrodifference equations \cite{Lu}, which capture discrete-time dispersal processes, and lattice equations, which model dispersal processes in discrete space habitats. One of the classical reaction-diffusion models is the Fisher-KPP type equation \cite{F, KPP} 
\begin{equation}
	\label{eq1.0}
	v_t=v_{xx}+f(v), \ x\in\mathbb{R},\ t>0,
\end{equation}
where the growth function $f$ satisfies the following conditions
\begin{equation}
	\label{eq1.3}
	\begin{aligned}
		f\in C^1([0,1]), \ f(0)=f(1)=0, \ f(s)>0, \ f'(0)>0 \ \text{and} \ f(s)\leq f'(0)s, \ s\in (0,1).
	\end{aligned}	
\end{equation}
This equation has been applied to study a variety of spreading phenomena in ecology and biology \cite{M}. A counterpart lattice equation of \eqref{eq1.0} can be derived by discretizing the spatial habitat when individuals move along a one-dimensional line grid as follows 
\begin{equation}
	\label{eq1.1}
	v'_i=v_{i+1}+v_{i-1}-2v_i+f(v_i), \ i\in\mathbb{Z}.
\end{equation}
For \eqref{eq1.1} with the Fisher-KPP type nonlinearity, Zinner et al.~\cite{ZHH} proved the existence and nonexistence of traveling waves. Besse et al.  recently~\cite{BFRZ} obtained the logarithmic shift for the exact location of solution level sets. For more general monostable lattice differential equations (LDEs), Chen and Guo \cite{CG, CG2} established the existence, uniqueness and stability of traveling waves. We refer to~\cite{AFP, LZ, WHW, WZ1} for extensions with time delay or non-local interactions and~\cite{GH, PW} for extensions with heterogeneous environments. When $f$ is of bistable type, research on traveling waves in LDEs includes their existence and stability~\cite{CMS, FZ, WH, Z1, Z2}, as well as the occurrence of the pinning phenomenon~\cite{BC, K}. For one-dimensional LDEs with multiple variables, propagation dynamics have been studied in various contexts, including epidemiological models~\cite{CGH}, the Lotka--Volterra model~\cite{GW1}, and the FitzHugh--Nagumo model~\cite{HS, SH}.

For two-dimensional LDEs, Guo and Wu~\cite{GW2} studied the following equation on a square lattice:
\begin{equation}
	\label{eq1.4}
	v'_{i,j}=v_{i+1,j}+v_{i-1,j}+v_{i,j+1}+v_{i,j-1}-4v_{i,j}+f(v_{i,j}), \ (i,j)\in\mathbb{Z}^2,
\end{equation}
where the diffusion happens between the nearest four neighbors (the von Neumann neighborhood). With a monostable structure, they proved the existence and uniqueness of traveling waves of \eqref{eq1.4}. With a bistable structure, the existence, uniqueness, and stability of traveling waves have been investigated in~\cite{CMV, HHV, S}. We also refer to ~\cite{CLW, CLW2} for extensions with time delay or non-local interactions. Building on the square lattice as a fundamental structure, subsequent studies have explored other spatial configurations on two-dimensional lattices, including periodic lattices~\cite{DLZ, GW}, striped lattices~\cite{WWR}, and time-periodically shifting habitats~\cite{GS}. Furthermore, some studies have focused on LDEs with more complex individual dispersals on two-dimensional lattices. For example, Hupkes and Van Vleck \cite{HV} established the traveling waves of lattice systems with a mixed operator $p\Delta_++q\Delta_{\times}$, where $p<0\leq q$, $\Delta_+$ denotes the dispersal associated with the nearest-neighbor lattices on the horizontal and vertical directions, and $\Delta_{\times}$ denotes dispersal associated with the next-nearest-neighbor lattices on the diagonal directions. We also refer to \cite{BF} for the traveling waves on monostable equations on two-dimensional lattice with a tree-like structure, where the spatial structure influences the minimal wave speed. For bistable equations on a tree-like lattice, Lou and Morita~\cite{LM} proved the spreading-transition-vanishing trichotomy of solutions.

The hexagonal structure, another typical kind of two-dimensional lattice, is often observed and ubiquitous in many biological systems, which is expected to influence the movement behavior of individuals and distribution of the population. For instance, honeybees construct their hives using hexagonal cells, which may enhance storage capacity and provide sufficient space for their brood \cite{YDKZGD}. Similarly, cells in many plant and animal tissues are arranged in an approximately hexagonal structure \cite{CAME, L}. Certain corals, such as brain corals, exhibit hexagonal patterns in their skeletal structures, believed to optimize space and structural integrity \cite{BH}. Propagation phenomena may also occur on spatial structures with hexagons, such as the transmission of chemical signals through cells~\cite{HFMMM}. Few models associated with such structures, both at macro and micro scales, have been proposed across various scientific fields. Spatial dynamics in hexagonal neighborhoods have been  simulated using cellular automata \cite{KK} and individual-based models \cite{FJS}. Furthermore, colonization population dynamics on hexagonal lattices have been explored through stochastic cellular automata, formulated by stochastic process arguments \cite{KRW, RWK}.

For general LDEs, the propagation dynamics exhibits periodicity when the underlying lattice has a symmetric structure. For monostable equations on a square lattice, numerical results have shown that the minimal wave speed attains its maximum along the lattice lines and its minimum across the lattice lines~\cite{CLW}. However, the dependence of the minimal wave speed on the propagation angle remains not fully understood. It is therefore of interest to investigate how the nonlinearity and distribution of nodes in admissible directions influence the propagation dynamics of LDEs with various spatial structures. Motivated by these considerations, we investigate the propagation dynamics on the hexagonal lattice, with particular emphasis on the angular dependence of the minimal wave speed. We point out that our analysis demonstrates that the minimal wave speed is monotonic between the directions corresponding to its maximum and minimum values, indicating that the distribution of nodes along different admissible directions does not sensitively influence the wave speed. However, when the nonlinearity is of bistable type, this dependence can become highly complex and counterintuitive according to some numerical observations, which will be explored in a future work.

The rest of this paper is organized as follows. The model formulation and main theorems are presented in Section \ref{sec:Model}. The existence and uniqueness of the traveling waves are established in Sections~\ref{se2} and \ref{se3}, respectively, using methods developed in \cite{CC2, CG2, DK, GW2, WZ2, WZ}. By applying the implicit function theorem and conducting careful calculations based on series expansions, we establish the monotonicity of the minimal wave speed with respect to angle within a period in Section~\ref{se4}, which constitutes the main motivation of the current study. Section~\ref{sec:Nume} presents numerical results for both the minimal wave speed and the spreading speed, with the latter measured from a Cauchy problem with compactly supported initial data. A discussion in Section~\ref{sec:Dis} concludes this paper.

\section{The model and key mathematical results}\label{sec:Model}
The model can be formulated by integrating the population growth process with dispersal in a hexagonal lattice. To begin with, we provide a schematic description of the hexagonal lattice, which consists of nodes and edges connecting the nearest adjacent nodes, as illustrated in Figure~\ref{fig1}.
\begin{figure}[h]
	\centering
	\begin{tikzpicture}[>=stealth,thick]
		\node (s00) at (0,0) [circle, draw=black]{};
		\node (s10) at (2,0) [circle, draw, fill=blue!50]{};
		\node (s20) at (4,0) [circle, draw, fill=red!50]{};
		\draw [dashed,] (s20)--(4.5,-0.866); \draw [dashed,] (s20)--(4.5,0.866); \draw [dashed,] (s20)--(5,0);
		\node (s-10) at (-2,0) [circle, draw, fill=blue!50]{};
		\node (s-20) at (-4,0) [circle, draw, fill=red!50]{};
		\draw [dashed,] (s-20)--(-4.5,-0.866); \draw [dashed,] (s-20)--(-4.5,0.866); \draw [dashed,] (s-20)--(-5,0);
		\node (s-11) at (-3,1.732) [circle, draw, fill=red!50]{};
		\draw [dashed,](s-11)--(-3.5,2.598); \draw [dashed,](s-11)--(-4,1.732);
		\node (s01) at (-1,1.732) [circle, draw, fill=blue!50]{};
		\node (s11) at (1,1.732) [circle, draw, fill=blue!50]{};
		\node (s21) at (3,1.732) [circle, draw, fill=red!50]{};
		\draw [dashed,](s21)--(3.5,2.598); \draw [dashed,](s21)--(4,1.732);
		\node (s02) at (-2,3.464) [circle, draw, fill=red!50]{};
		\draw [dashed,](s02)--(-1.5,4.33); \draw [dashed,](s02)--(-2.5,4.33); \draw [dashed,](s02)--(-3,3.464);
		\node (s12) at (0,3.464) [circle, draw, fill=red!50]{};
		\draw [dashed,] (s12)--(-0.5,4.33); \draw [dashed,] (s12)--(0.5,4.33); 
		\node (s22) at (2,3.464) [circle, draw, fill=red!50]{};
		\draw [dashed,](s22)--(1.5,4.33); \draw [dashed,](s22)--(2.5,4.33); \draw [dashed,](s22)--(3,3.464);
		\node (s-2-1) at (-3,-1.732) [circle, draw, fill=red!50]{};
		\draw [dashed,](s-2-1)--(-3.5,-2.598); \draw [dashed,](s-2-1)--(-4,-1.732);
		\node (s-1-1) at (-1,-1.732) [circle, draw, fill=blue!50]{};
		\node (s0-1) at (1,-1.732) [circle, draw, fill=blue!50]{};
		\node (s1-1) at (3,-1.732) [circle, draw, fill=red!50]{};
		\draw [dashed,](s1-1)--(3.5,-2.598); \draw [dashed,](s1-1)--(4,-1.732);
		\node (s-2-2) at (-2,-3.464) [circle, draw, fill=red!50]{};
		\draw [dashed,](s-2-2)--(-1.5,-4.33); \draw [dashed,](s-2-2)--(-2.5,-4.33); \draw [dashed,](s-2-2)--(-3,-3.464);
		\node (s-1-2) at (0,-3.464) [circle, draw, fill=red!50]{};
		\draw [dashed,] (s-1-2)--(-0.5,-4.33); \draw [dashed,] (s-1-2)--(0.5,-4.33); 
		\node (s0-2) at (2,-3.464) [circle, draw, fill=red!50]{};
		\draw [dashed,](s0-2)--(1.5,-4.33); \draw [dashed,](s0-2)--(2.5,-4.33); \draw [dashed,](s0-2)--(3,-3.464);
		\draw [-] (s02)--(s12)node[]{};
		\draw [-] (s12)--(s22)node[]{};
		\draw [-] (s-11)--(s01)node[]{};
		\draw [-] (s01)--(s11)node[]{};
		\draw [-] (s11)--(s21)node[]{};
		\draw [-] (s-20)--(s-10)node[]{};
		\draw [-] (s-10)--(s00)node[]{};
		\draw [-] (s00)--(s10)node[]{};
		\draw [-] (s10)--(s20)node[]{};
		\draw [-] (s-2-1)--(s-1-1)node[]{};
		\draw [-] (s-1-1)--(s0-1)node[]{};
		\draw [-] (s0-1)--(s1-1)node[]{};
		\draw [-] (s-2-2)--(s-1-2)node[]{};
		\draw [-] (s-1-2)--(s0-2)node[]{};
		\draw [-] (s02)--(s-11)node[]{};
		\draw [-] (s-11)--(s-20)node[]{};
		\draw [-] (s12)--(s01)node[]{};
		\draw [-] (s01)--(s-10)node[]{};
		\draw [-] (s-10)--(s-2-1)node[]{};
		\draw [-] (s22)--(s11)node[]{};
		\draw [-] (s11)--(s00)node[]{};
		\draw [-] (s00)--(s-1-1)node[]{};
		\draw [-] (s-1-1)--(s-2-2)node[]{};
		\draw [-] (s21)--(s10)node[]{};
		\draw [-] (s10)--(s0-1)node[]{};
		\draw [-] (s0-1)--(s-1-2)node[]{};
		\draw [-] (s20)--(s1-1)node[]{};
		\draw [-] (s1-1)--(s0-2)node[]{};
		\draw [-] (s22)--(s21)node[]{};
		\draw [-] (s21)--(s20)node[]{};
		\draw [-] (s12)--(s11)node[]{};
		\draw [-] (s11)--(s10)node[]{};
		\draw [-] (s10)--(s1-1)node[]{};
		\draw [-] (s02)--(s01)node[]{};
		\draw [-] (s01)--(s00)node[]{};
		\draw [-] (s00)--(s0-1)node[]{};
		\draw [-] (s0-1)--(s0-2)node[]{};
		\draw [-] (s-11)--(s-10)node[]{};
		\draw [-] (s-10)--(s-1-1)node[]{};
		\draw [-] (s-1-1)--(s-1-2)node[]{};
		\draw [-] (s-20)--(s-2-1)node[]{};
		\draw [-] (s-2-1)--(s-2-2)node[]{};
		\node at (0.5,-0.2) []{\tiny$(0,0)$};
		\node at (2.5,-0.2) []{\tiny$(1,0)$};
		\node at (4.5,-0.2) []{\tiny$(2,0)$};
		\node at (-2.6,-0.2) []{\tiny$(-1,0)$};
		\node at (-4.6,-0.2) []{\tiny$(-2,0)$};
		\node at (-3.8,1.5) []{\tiny$(-\frac{3}{2},\frac{\sqrt{3}}{2})$};
		\node at (-1.8,1.5) []{\tiny$(-\frac{1}{2},\frac{\sqrt{3}}{2})$};
		\node at (1.7,1.5) []{\tiny$(\frac{1}{2},\frac{\sqrt{3}}{2})$};
		\node at (3.7,1.5) []{\tiny$(\frac{3}{2},\frac{\sqrt{3}}{2})$};
		\node at (-2.75,3.25) []{\tiny$(-1,\sqrt{3})$};
		\node at (0.65,3.25) []{\tiny$(0,\sqrt{3})$};
		\node at (2.65,3.25) []{\tiny$(1,\sqrt{3})$};
		\node at (-3.95,-2) []{\tiny$(-\frac{3}{2},-\frac{\sqrt{3}}{2})$};
		\node at (-1.95,-2) []{\tiny$(-\frac{1}{2},-\frac{\sqrt{3}}{2})$};
		\node at (1.85,-2) []{\tiny$(\frac{1}{2},-\frac{\sqrt{3}}{2})$};
		\node at (3.85,-2) []{\tiny$(\frac{3}{2},-\frac{\sqrt{3}}{2})$};	
		\node at (-2.9,-3.7) []{\tiny$(-1,-\sqrt{3})$};
		\node at (0.8,-3.7) []{\tiny$(0,-\sqrt{3})$};
		\node at (2.8,-3.7) []{\tiny$(1,-\sqrt{3})$};	
		\draw [->](-7,-3.5)--(-7,-1.5)node[left=0.05cm]{\footnotesize$y$};
		\draw [->](-8,-2.5)--(-6,-2.5)node[right=0.05cm]{\footnotesize$x$};				
	\end{tikzpicture}
	\caption{A hexagonal lattice described by the coordinates $x$-$y$.}\label{fig1}		
\end{figure}
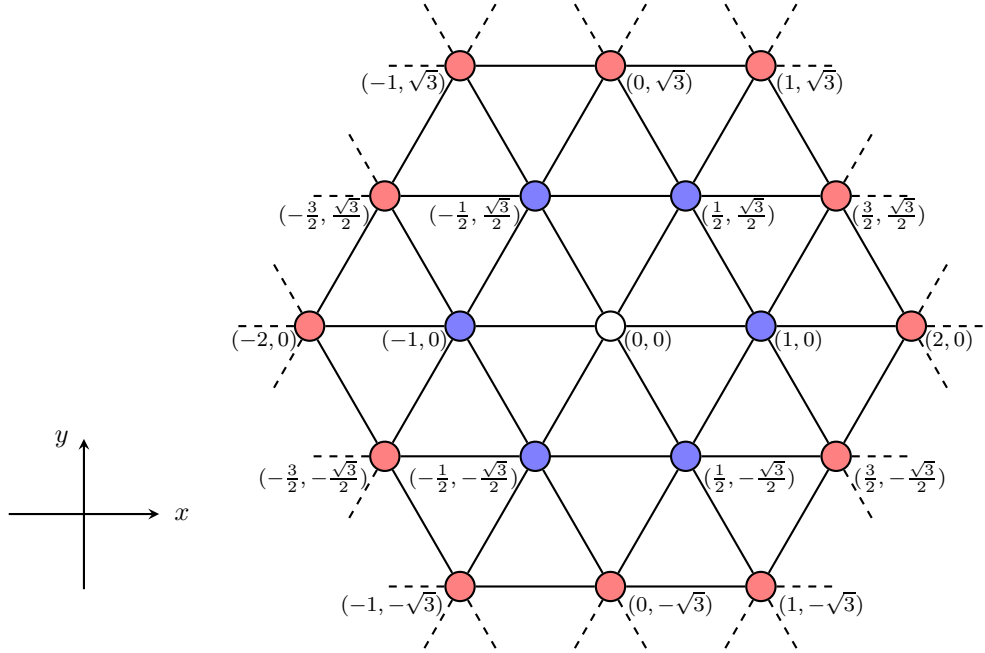

We denote the hexagonal lattice by  
\begin{equation*}
	\mathbb{H}:=\bigcup_{i,j\in\mathbb{Z}}\left((i,\sqrt{3}j)\cup(\frac{1}{2}+i,\frac{\sqrt{3}}{2}+\sqrt{3}j)\right).
\end{equation*}
Note that each node in the hexagonal lattice has six nearest neighbors. We assume that movement occurs only between these nearest-neighbor nodes, and the probability of the movement to each nearest-neighbor node is exactly the same, being $1/6$. Then, the movement can be described by an operator $\Delta_h$ acting on a function $u:\mathbb{H}\times\mathbb{R}\rightarrow\mathbb{R}$, defined as
\begin{equation*}
	\label{eq1.6}
	\begin{aligned}
		\frac{1}{6}\Delta_h[u](x,y,\cdot):=&\frac{1}{6}u(x+1,y,\cdot)+\frac{1}{6}u(x-1,y,\cdot)+\frac{1}{6}u(x+\frac{1}{2},y+\frac{\sqrt{3}}{2},\cdot)\\
		&+\frac{1}{6}u(x-\frac{1}{2},y-\frac{\sqrt{3}}{2},\cdot)+\frac{1}{6}u(x+\frac{1}{2},y-\frac{\sqrt{3}}{2},\cdot)+\frac{1}{6}u(x-\frac{1}{2},y+\frac{\sqrt{3}}{2},\cdot)\\
		&-u(x,y,\cdot), \ (x,y)\in\mathbb{H}.
	\end{aligned}		
\end{equation*}
Please note that this operator is derived by using the $x$-$y$ coordinates illustrated in Figure~\ref{fig1}. Combining with the population birth-death process, we have the following model
\begin{equation}
	\label{eq1.5}
	u'(x,y,t)=\frac{1}{6}\Delta_h[u](x,y,t)+f(u(x,y,t)), \ (x,y)\in \mathbb{H}, \ t>0.
\end{equation}
The population growth function $f$ is the Fisher-KPP type, that is, $f$ satisfies \eqref{eq1.3}. Moreover, we assume that  there exists a constant $N\geq0$ and $\theta\in(0,1]$ such that 
\begin{equation}
	\label{eq3.2}
	f'(0)s-N s^{1+\theta}\leq f(s), \ s\in[0,1].
\end{equation}

We are interested in traveling waves of \eqref{eq1.5}. For some function $U\in C^1(\mathbb{R})$ and some constant $c\in\mathbb{R}$, we may consider solutions of the form 
\begin{equation*}
	u(x,y,t)=U(ct-(x,y)\cdot(\kappa,\sigma))		
\end{equation*}	
to be the traveling waves of \eqref{eq1.5}, where $(\kappa,\sigma):=(\cos\alpha,\sin\alpha)$ , $\alpha\in[0,2\pi)$, is the direction on $\mathbb{H}$. Here, Fisher-KPP waves are defined as $U(ct-x\kappa-y\sigma)$ satisfying $U(-\infty)=0$, $U(+\infty)=1$ and $U\in[0,1]$. Combining \eqref{eq1.5} with \eqref{eq1.3}, we obtain 
\begin{equation}
	\label{eq1.17}
	\left\{\begin{aligned}
		&cU'(\xi)=\frac{1}{6}\mathcal{D}_h[U](\xi)+f(U(\xi)),&&\xi\in\mathbb{R},\\
		&U(\xi)\in[0,1], &&\xi\in\mathbb{R},\\
		&\lim_{\xi\rightarrow-\infty}U(\xi)=0, \ \lim_{\xi\rightarrow+\infty}U(\xi)=1, &&\xi\in\mathbb{R},
	\end{aligned}\right.
\end{equation}
where $\mathcal{D}_h: C^1(\mathbb{R})\rightarrow C^1(\mathbb{R})$ is a linear operator given by
\begin{equation*}
	\begin{aligned}		
		\mathcal{D}_h[U](\xi)=&U(\xi-\kappa)+U(\xi+\kappa)+U(\xi-\frac{1}{2}\kappa-\frac{\sqrt{3}}{2}\sigma)+U(\xi+\frac{1}{2}\kappa+\frac{\sqrt{3}}{2}\sigma)\\
		&+U(\xi-\frac{1}{2}\kappa+\frac{\sqrt{3}}{2}\sigma)+U(\xi+\frac{1}{2}\kappa-\frac{\sqrt{3}}{2}\sigma)-6U(\xi).
	\end{aligned}	
\end{equation*}

The existence of the traveling waves is solely dependent on a constant $c^*$, defined as
\begin{equation}	
	\label{eq1.14}
	\begin{aligned}
		c^*:=&\inf_{\lambda>0}\frac{1}{\lambda}\Bigg[\frac{1}{6}\left(\mathrm{e}^{-\kappa\lambda}+\mathrm{e}^{\kappa\lambda}+\mathrm{e}^{-(\frac{1}{2}\kappa+\frac{\sqrt{3}}{2}\sigma)\lambda}+\mathrm{e}^{(\frac{1}{2}\kappa+\frac{\sqrt{3}}{2}\sigma)\lambda}+\mathrm{e}^{-(\frac{1}{2}\kappa-\frac{\sqrt{3}}{2}\sigma)\lambda}+\mathrm{e}^{(\frac{1}{2}\kappa-\frac{\sqrt{3}}{2}\sigma)\lambda}-6\right)\\
		&+f'(0)\Bigg]:=\inf_{\lambda>0}\frac{g(\lambda)}{\lambda}.
	\end{aligned}	
\end{equation}
Clearly, for all $\alpha\in[0,2\pi)$, $g(\lambda)$ is convex with respect to $\lambda\in(0,+\infty)$, and $g(\lambda)\geq g(0)=f'(0)>0$, which implies that $c^*$ is a well-defined real number, and $c^*>0$.

Next, we present the main results of this paper.
\begin{theorem}[Existence of traveling waves]
	\label{th1.2}
	Problem \eqref{eq1.17} admits a solution $(c,U)$ if and only if $c\geq c^*$.
\end{theorem}

\begin{theorem}[Uniqueness of traveling waves]
	\label{th1.4}
	For $c\geq c^*$, solutions of \eqref{eq1.17} are unique up to a translation.
\end{theorem} 

\begin{theorem}[Angle-dependent speed]
	\label{th1.1}
	The quantity $c^*=c^*(\alpha)$ is $\frac{\pi}{3}$-periodic with $\alpha\in[0,2\pi)$, monotonically decreasing in $\alpha\in[0,\frac{\pi}{6}]$ and monotonically increasing in $\alpha\in[\frac{\pi}{6},\frac{\pi}{3}]$.	
\end{theorem}

By numerically calculating $c^*$ and solving a Cauchy problem with compact supported initial values, we further observe that the spreading speed coincides with the minimal wave speed in all directions.

\section{The existence of traveling waves}\label{se2}

This section is devoted to proving Theorem~\ref{th1.2}, which describes the existence of traveling waves for equation \eqref{eq1.17}. To achieve this, we use methods developed in \cite{CG2, GW2, WZ2, WZ} to establish a set of lemmas and organize them into the preliminaries. Subsequently, we use the lemmas to prove the sufficiency and necessity conditions outlined in Theorem~\ref{th1.2}.

\subsection{Preliminaries}\label{se21}

Note that the expressions for $\mathcal{D}_h[U](\xi)$ and $g(\lambda)$ are quite lengthy. To facilitate the analysis of problems involving  terms $\pm\kappa$, $\pm(\frac{1}{2}\kappa+\frac{\sqrt{3}}{2}\sigma)$ and $\pm(\frac{1}{2}\kappa-\frac{\sqrt{3}}{2}\sigma)$, we introduce the following notations:
\begin{equation}
	\label{eq3.3}
	\delta_1:=\kappa, \ \delta_2:=\frac{1}{2}\kappa+\frac{\sqrt{3}}{2}\sigma,\  \delta_3:=\frac{1}{2}\kappa-\frac{\sqrt{3}}{2}\sigma.
\end{equation}
Clearly, $\delta_1, \delta_2, \delta_3\in[-1,1]$. 

To prove the sufficiency, we need an equivalent form of the first equation in \eqref{eq1.17}. By integrating the first equation in \eqref{eq1.17} form $-\infty$ to $+\infty$, it follows that $c=\int_{-\infty}^{+\infty}f(U(\xi))\mathrm{d}\xi>0$ for $(c,U(\xi))$ satisfying \eqref{eq1.17}. Therefore, we introduce an operator $\mathscr{H}^\mu:C(\mathbb{R})\rightarrow C(\mathbb{R})$ by
\begin{equation}
	\label{eq3.8}	
	\mathscr{H}^\mu[U](\xi)=\mu U(\xi)+\frac{1}{6c}\mathcal{D}_h[U](\xi)+\frac{1}{c}f(U(\xi)),
\end{equation}
where $\mu>\frac{1}{c}\max_{0\leq U\leq1}\{|f'(U)|+1\}$. If $U_1(\xi)\leq U_2(\xi)$, $\xi\in\mathbb{R}$, then there exists $\tilde{U}(\xi)\in [U_1(\xi), U_2(\xi)]$ such that
\begin{equation}
	\begin{aligned}	
		\label{eq3.9}
		&\mathscr{H}^\mu[U_1](\xi)-\mathscr{H}^\mu[U_2](\xi)\\
		=&\mu\left(U_1(\xi)-U_2(\xi)\right)+\frac{1}{6c}(\mathcal{D}_h[U_1](\xi)-\mathcal{D}_h[U_2](\xi))+\frac{1}{c}\left(f(U_1(\xi))-f(U_2(\xi))\right)\\
		=&\left(U_1(\xi)-U_2(\xi)\right)\Bigg[\mu-\frac{1}{c}+\frac{1}{c}f'(\tilde{U}(\xi))+\frac{1}{6c}\sum_{m=1}^{3}\Bigg(\frac{U_1(\xi+\delta_m)-U_2(\xi+\delta_m)}{U_1(\xi)-U_2(\xi)}\\
		&+\frac{U_1(\xi-\delta_m)-U_2(\xi-\delta_m)}{U_1(\xi)-U_2(\xi)}\Bigg)\Bigg]\\
		\leq&0.
	\end{aligned}		
\end{equation}	
In addition, multiplying the two sides of the first equation of \eqref{eq1.17} by $\mathrm{e}^{\mu\xi}$, we have
\begin{equation*}		
	\mathrm{e}^{\mu\xi}U'(\xi)=\mathrm{e}^{\mu\xi}\frac{1}{6c}\mathcal{D}_h[U](\xi)+\mathrm{e}^{\mu\xi}\frac{1}{c}f(U(\xi)).
\end{equation*}	
Combining with \eqref{eq3.8}, we integrate the above equation from $-\infty$ to $\xi$ to get the equivalent form of the first equation of \eqref{eq1.17} with $U(+\infty)=1$, that is,
\begin{equation}	
	\label{eq3.10}	
	U(\xi)=\mathrm{e}^{-\mu\xi}\int_{-\infty}^{\xi}\mathrm{e}^{\mu\zeta}\mathscr{H}^\mu[U](\zeta)\mathrm{d}\zeta.
\end{equation}	

The purpose of obtaining the expression \eqref{eq3.10} is to use monotone iteration method developed in \cite{WZ2}, where iterative initial values are acted as the upper and lower solutions. According to \cite{WZ}, we define the upper and lower solutions of \eqref{eq1.17}. We call a pair of continuous functions $(\overline{U}(\xi), \underline{U}(\xi))$ upper and lower solutions of \eqref{eq1.17} if $\overline{U}'$ and $\underline{U}'$ exist on $\mathbb{R}\backslash Y$, and 
\begin{equation}
	\label{eq3.1}
	\left\{ \begin{aligned}		
		&c\overline{U}'(\xi)\geq \frac{1}{6}\mathcal{D}_h[\overline U](\xi)+f(\overline{U}(\xi)), && \xi\in\mathbb{R}\backslash Y,\\
		&c\underline{U}'(\xi)\leq \frac{1}{6}\mathcal{D}_h[\underline U](\xi)+f(\underline{U}(\xi)), && \xi\in\mathbb{R}\backslash Y,\\
		&0\leq\underline{U}(\xi)\leq\overline{U}(\xi)\leq1, && \xi\in\mathbb{R}\backslash Y,
	\end{aligned}\right.	
\end{equation}		
where $\overline{U}(\xi)$ is non-decreasing and $\underline{U}(\xi)\not\equiv0$ for $\xi\in\mathbb{R}\backslash Y$, $Y$ is a set with only a finite number of points.  Additionally, we provide a proposition that is useful for constructing the upper and lower solutions of equation \eqref{eq1.17}. The proof of this proposition is straightforward and is therefore omitted.
\begin{proposition}
	\label{pr3.1}
	Recall that $g(\lambda)$ and $c^*$ defined in \eqref{eq1.14}. Let $G(c,\lambda):=c\lambda-g(\lambda)$. Then there exists $\lambda^*>0$ such that $G(c^*,\lambda^*)=0$, and there exist $\lambda_1$ and $\lambda_2$ satisfying $0<\lambda_1\leq\lambda^*\leq\lambda_2$ such that the following statements hold:
	\begin{itemize}
		\item[$\mathrm{(i)}$] If $c>c^*$, then $G(c,\lambda_1)=G(c,\lambda_2)=0<G(c,\lambda^*)$, where $\lambda_1<\lambda^*<\lambda_2$.
		\item[$\mathrm{(ii)}$] If $c=c^*$, then $G(c,\lambda_1)=G(c,\lambda_2)=G(c,\lambda^*)=0$, where $\lambda_1=\lambda^*=\lambda_2$.
		\item[$\mathrm{(iii)}$] If $c<c^*$, then $G(c,\lambda)<0$, where $\lambda>0$.
	\end{itemize}	
\end{proposition}

Now, we give a lemma about the existences of upper and lower solutions for \eqref{eq1.17}.
\begin{lemma}
	\label{le3.2}
	Assume \eqref{eq1.3} and \eqref{eq3.2} hold. For $c>c^*$, let
	\begin{equation*}
		\overline{U}(\xi):=\min\{1,\mathrm{e}^{\lambda_1\xi}\}, \ \underline{U}(\xi):=\max\left\{0,\left(1-M\mathrm{e}^{\gamma\xi}\right)\mathrm{e}^{\lambda_1\xi}\right\}, \ \xi\in\mathbb{R},	
	\end{equation*}
	where $M\geq \frac{N}{G(c,\lambda_1+\gamma)}$ and $0<\gamma<\min\{\lambda_1\theta, \lambda_2-\lambda_1\}$. Then $\overline{U}$ and $\underline{U}$ are the upper and lower solutions of \eqref{eq1.17}, respectively. 
\end{lemma}
\begin{proof}
	When $\xi\geq0$, we have $\overline{U}(\xi)=1$, and
	\begin{equation*}
		\overline{U}(\xi\pm\delta_m)=\min\left\{1,\mathrm{e}^{\lambda_1(\xi\pm\delta_m)}\right\}\left\{ \begin{aligned}		
			&=1,&&\xi\pm\delta_m\geq0,\\
			&\in(0,1],&&\xi\pm\delta_m\leq0,
		\end{aligned}\right.	
	\end{equation*}
	where $m=1,2,3$. It follows from \eqref{eq1.3} that
	\begin{equation*}
		\begin{aligned}		
			c\overline{U}'(\xi)-\frac{1}{6}\mathcal{D}_h[\overline{U}](\xi)-f(\overline{U}(\xi))
			=-\frac{1}{6}\left[\sum_{m=1}^{3}\left(\overline{U}(\xi-\delta_m)+\overline{U}(\xi+\delta_m)\right)-6\overline{U}(\xi)\right]
			\geq&0. 
		\end{aligned}	
	\end{equation*}
	
	When $\xi<0$, we have
	$\overline{U}(\xi)=\mathrm{e}^{\lambda_1\xi}$. By \eqref{eq3.2} and $G(c,\lambda_1)=0$, we obtain
	\begin{equation*}
		\begin{aligned}		
			&c\overline{U}'(\xi)-\frac{1}{6}\mathcal{D}_h[\overline{U}](\xi)-f(\overline{U}(\xi))\\
			\geq&c\overline{U}'(\xi)-\frac{1}{6}\mathcal{D}_h[\overline{U}](\xi)-f'(0)\mathrm{e}^{\lambda_1\xi}\\
			=&\lambda_1c\mathrm{e}^{\lambda_1\xi}-\frac{1}{6}\left[\sum_{m=1}^{3}\left(\mathrm{e}^{\lambda_1(\xi-\delta_m)}+\mathrm{e}^{\lambda_1(\xi+\delta_m)}\right)-6\mathrm{e}^{\lambda_1\xi}\right]-f'(0)\mathrm{e}^{\lambda_1\xi}\\
			=&G(c,\lambda_1)\mathrm{e}^{\lambda_1\xi}\\
			=&0. 
		\end{aligned}	
	\end{equation*}
	In summary,  $\overline{U}(\xi)=\min\{1,\mathrm{e}^{\lambda_1\xi}\}$ is an upper solution of \eqref{eq1.17}.
	
	When $\xi\geq-\frac{\ln M}{\gamma}$, we have $\underline{U}(\xi)=0$, and
	\begin{equation*}
		\underline{U}(\xi\pm\delta_m)=\max\left\{0,\left(1-M\mathrm{e}^{\gamma(\xi\pm\delta_m)}\right)\mathrm{e}^{\lambda_1(\xi\pm\delta_m)}\right\}\left\{ \begin{aligned}		
			&=0,&&\xi\pm\delta_m\geq-\frac{\ln M}{\gamma},\\
			&\in[0,1),&&\xi\pm\delta_m\leq-\frac{\ln M}{\gamma},
		\end{aligned}\right.
	\end{equation*}
	where $m=1,2,3$. Together with \eqref{eq1.3}, we obtain $c\underline{U}'(\xi)-\frac{1}{6}\mathcal{D}_h[\underline{U}](\xi)-f(\underline{U}(\xi))\leq0$. 
	
	When $\xi<-\frac{\ln M}{\gamma}$, we have
	$\underline{U}(\xi)=\left(1-M\mathrm{e}^{\gamma\xi}\right)\mathrm{e}^{\lambda_1\xi}$. It follows from (i) in Proposition~\ref{pr3.1}  that $G(c,\lambda_1)=0<G(c,\lambda_1+\gamma)$, which, combining \eqref{eq3.2} with $M\geq\frac{N}{G(c,\lambda_1+\gamma)}$, implies 
	\begin{equation*}
		\begin{aligned}		
			&c\underline{U}'(\xi)-\frac{1}{6}\mathcal{D}_h[\underline{U}](\xi)-f(\underline{U}(\xi))\\
			\leq&c\underline{U}'(\xi)-\frac{1}{6}\mathcal{D}_h[\underline{U}](\xi)-f'(0)\underline{U}(\xi)+N(\underline{U}(\xi))^{1+\theta}\\
			=&\left[c\lambda_1-\frac{1}{6}\sum_{m=1}^{3}\left(\mathrm{e}^{\lambda_1\delta_m}+\mathrm{e}^{-\lambda_1\delta_m}\right)+1-f'(0)\right]\mathrm{e}^{\lambda_1\xi}\\
			&+\left[-cM(\lambda_1+\gamma)+\frac{1}{6}M\sum_{m=1}^{3}\left(\mathrm{e}^{(\lambda_1+\gamma)\delta_m}+\mathrm{e}^{-(\lambda_1+\gamma)\delta_m}\right)-1+f'(0)M\right]\mathrm{e}^{(\lambda_1+\gamma)\xi}\\
			&+N\left(1-M\mathrm{e}^{\gamma\xi}\right)^{1+\theta}\mathrm{e}^{\lambda_1\xi(1+\theta)}\\
			\leq&G(c,\lambda_1)\mathrm{e}^{\lambda_1\xi}-MG(c,\lambda_1+\gamma)\mathrm{e}^{(\lambda_1+\gamma)\xi}+N\mathrm{e}^{(\lambda_1+\theta\lambda_1)\xi}\\
			\leq&-M G(c,\lambda_1+\gamma)\mathrm{e}^{(\lambda_1+\gamma)\xi}+N\mathrm{e}^{(\lambda_1+\gamma)\xi}\\
			\leq&0.	  
		\end{aligned}			
	\end{equation*}
	Thus, we conclude that $\underline{U}(\xi)=\max\{0,\left(1-M\mathrm{e}^{\gamma\xi}\right)\mathrm{e}^{\lambda_1\xi}\}$ is a lower solution of \eqref{eq1.17}. This completes the proof.		
\end{proof}

To prove the necessity, we present the following lemma.
\begin{lemma}
	\label{le3.3}
	Let $(c,U(\xi))$ be a solution of \eqref{eq1.17}. Then $\frac{U(\xi+\delta)}{U(\xi)}$ is uniformly bounded for $\xi\in\mathbb{R}$, $\delta\in[-1,1]$.
\end{lemma}
\begin{proof}
	Let $I(\xi):=U(\xi)\mathrm{e}^{\frac{\xi}{c}}$, $\xi\in\mathbb{R}$.  It follows from the first formula of \eqref{eq1.17} that
	\begin{equation*}
		\begin{aligned}		
			0=cU'(\xi)+U(\xi)-\frac{1}{6}\sum_{m=1}^{3}\left(U(\xi+\delta_m)+U(\xi-\delta_m)\right)-f(U(\xi))\leq c\left(U'(\xi)+\frac{1}{c}U(\xi)\right),	
		\end{aligned}				
	\end{equation*}
	which implies $I'(\xi)=\left(U'(\xi)+\frac{1}{c} U(\xi)\right)\mathrm{e}^{\frac{\xi}{c}}\geq0$. Thus, we have $I(\xi-\delta)\leq I(\xi)$ for $\delta\geq0$. That is, $U(\xi-\delta)\mathrm{e}^{\frac{\xi-\delta}{c}}\leq U(\xi)\mathrm{e}^{\frac{\xi}{c}}$. Therefore we get
	\begin{equation}
		\label{eq3.4}
		\frac{U(\xi-\delta)}{U(\xi)}\leq\mathrm{e}^{\frac{\delta}{c}}, \ \xi\in\mathbb{R}, \ \delta\geq0.
	\end{equation}
	
	For $\delta\geq0$, if there exists $\zeta$ such that $\xi-\frac{\delta}{2}\leq\zeta\leq\xi$, then, applying \eqref{eq3.4}, we have 
	\begin{equation}
		\label{eq3.5}
		U(\zeta+\delta)\mathrm{e}^{\frac{\delta}{2c}}\geq U(\zeta+\delta)\mathrm{e}^{\frac{1}{c}(\frac{\delta}{2}+\zeta-\xi)}\geq U(\zeta+\delta-(\frac{\delta}{2}+\zeta-\xi))=U(\xi+\frac{\delta}{2}).
	\end{equation}
	Moreover, if there exists $\zeta$ such that $\xi-\delta\leq\zeta\leq\xi$, then, by $I(\zeta)\leq I(\xi)$, we obtain
	\begin{equation}
		\label{eq3.6}	
		U(\zeta)\leq U(\xi)\mathrm{e}^{\frac{\xi-\zeta}{c}}\leq U(\xi)\mathrm{e}^{\frac{\delta}{c}}.
	\end{equation}
	Integrating the first equation of \eqref{eq1.17} from $-\infty$ to $\xi$, and using \eqref{eq3.5} and \eqref{eq3.6}, we have
	\begin{equation*}
		\begin{aligned}		
			cU(\xi)=&\int_{-\infty}^{\xi}\frac{1}{6}\mathcal{D}_h[U](\zeta)\mathrm{d}\zeta+\int_{-\infty}^{\xi}f(U(\zeta))\mathrm{d}\zeta\\
			\geq&\frac{1}{6}\int_{-\infty}^{\xi}\sum_{m=1}^{3}\left(U(\zeta+\delta_m)+U(\zeta-\delta_m)-6U(\zeta)\right)\mathrm{d}\zeta\\
			=&\frac{1}{6}\int_{\xi-|\delta_1|}^{\xi}\left(U(\zeta+|\delta_1|)-U(\zeta)\right)\mathrm{d}\zeta+\frac{1}{6}\int_{\xi-|\delta_2|}^{\xi}\left(U(\zeta+|\delta_2|)-U(\zeta)\right)\mathrm{d}\zeta\\
			&+\frac{1}{6}\int_{\xi-|\delta_3|}^{\xi}\left(U(\zeta+|\delta_3|)-U(\zeta)\right)\mathrm{d}\zeta\\
			\geq&\frac{1}{6}\left[\int_{\xi-\frac{|\delta_3|}{2}}^{\xi}U(\zeta+|\delta_3|)\mathrm{d}\zeta-\int_{\xi-|\delta_1|}^{\xi}U(\zeta)\mathrm{d}\zeta-\int_{\xi-|\delta_2|}^{\xi}U(\zeta)\mathrm{d}\zeta-\int_{\xi-|\delta_3|}^{\xi}U(\zeta)\mathrm{d}\zeta\right]\\
			\geq&\frac{1}{6}\left[\frac{|\delta_3|}{2}\mathrm{e}^{-\frac{|\delta_3|}{2c}}U(\xi+\frac{|\delta_3|}{2})-U(\xi)\left(|\delta_1|\mathrm{e}^{\frac{|\delta_1|}{c}}+|\delta_2|\mathrm{e}^{\frac{|\delta_2|}{c}}+|\delta_3|\mathrm{e}^{\frac{|\delta_3|}{c}}\right)\right],
		\end{aligned}		
	\end{equation*}
	which implies
	\begin{equation*}	
		\frac{U(\xi+\frac{|\delta_3|}{2})}{U(\xi)}\leq\frac{6c+|\delta_1|\mathrm{e}^{\frac{|\delta_1|}{c}}+|\delta_2|\mathrm{e}^{\frac{|\delta_2|}{c}}+|\delta_3|\mathrm{e}^{\frac{|\delta_3|}{c}}}
		{\frac{|\delta_3|}{2}\mathrm{e}^{-\frac{|\delta_3|}{2c}}}, \ \xi\in\mathbb{R}.
	\end{equation*}
	Subsequently, we can conclude from the above equation that $\frac{U(\xi+|\delta_3|)}{U(\xi)}$ is bounded for $\xi\in\mathbb{R}$. Thus, for $\delta\geq0$, we can show that  $\frac{U(\xi+\delta)}{U(\xi)}$ is uniformly bounded for $\xi\in\mathbb{R}$, which, together with \eqref{eq3.4}, implies that this lemma holds.	
\end{proof}

\subsection{The proof of Theorem~\ref{th1.2}}\label{se22} 

\begin{proof}	
	\textbf{The sufficiency}. 
	Based on the above preparation, we prove the sufficiency in case $c> c^*$ and case $c= c^*$.
	
	\emph{Case} $c> c^*$. From Lemma~\ref{le3.2}, we see that \eqref{eq1.17} admits a lower solution $\underline{U}(\xi)$ and an upper solution $\overline{U}(\xi)$. Define
	\begin{equation}	
		\label{eq3.12}	
		U_{n+1}(\xi):= \mathrm{e}^{-\mu\xi}\int_{-\infty}^{\xi}\mathrm{e}^{\mu\zeta}\mathscr{H}^\mu[U_{n}](\zeta)\mathrm{d}\zeta,\ n=1,2,3,\cdots,
	\end{equation}
	where
	\begin{equation}	
		\label{eq3.11}	
		U_1(\xi):=\mathrm{e}^{-\mu\xi}\int_{-\infty}^{\xi}\mathrm{e}^{\mu\zeta}\mathscr{H}^\mu[\overline{U}](\zeta)\mathrm{d}\zeta.
	\end{equation}
	By \eqref{eq3.10}, \eqref{eq3.1} and \eqref{eq3.11}, it follows that $\overline{U}(\xi)\geq\mathrm{e}^{-\mu\xi}\int_{-\infty}^{\xi}\mathrm{e}^{\mu\zeta}\mathscr{H}^\mu[\overline{U}](\zeta)\mathrm{d}\zeta=U_1(\xi)$. Since $\underline{U}(\xi)\leq\overline{U}(\xi)$, we obtain from $\eqref{eq3.9}$ that $\mathscr{H}^\mu[\underline{U}](\xi)\leq\mathscr{H}^\mu[\overline{U}](\xi)$. For $\underline{U}(\xi)$, we see from $\eqref{eq3.9}$ and $\eqref{eq3.10}$ that 
	\begin{equation*}	
		\begin{aligned}		
			\mathrm{e}^{\mu\xi}\underline{U}(\xi)\leq\int_{-\infty}^{\xi}\mathrm{e}^{\mu\zeta}\mathscr{H}^\mu[\underline{U}](\zeta)\mathrm{d}\zeta\leq\int_{-\infty}^{\xi}\mathrm{e}^{\mu\zeta}\mathscr{H}^\mu[\overline{U}](\zeta)\mathrm{d}\zeta=\mathrm{e}^{\mu\xi}U_1(\xi),
		\end{aligned}	
	\end{equation*}
	which means $U_1(\xi)\geq\underline{U}(\xi)$. Hence, we have $\underline{U}(\xi)\leq U_1(\xi)\leq\overline{U}(\xi)$, $\xi\in\mathbb{R}$. Moreover, taking the derivative of \eqref{eq3.11}, we get	
	\begin{equation*}	
		\begin{aligned}		
			U_1'(\xi)=&-\mu \mathrm{e}^{-\mu\xi}\int_{-\infty}^{\xi}\mathrm{e}^{\mu \zeta}\mathscr{H}^\mu[\overline{U}](\zeta)\mathrm{d}\zeta+\mathrm{e}^{-\mu\xi}\mathscr{H}^\mu[\overline{U}](\xi)\mathrm{e}^{\mu\xi}\\
			=&-\mu \mathrm{e}^{-\mu\xi}\int_{-\infty}^{\xi}\mathrm{e}^{\mu \zeta}\mathscr{H}^\mu[\overline{U}](\zeta)\mathrm{d}\zeta+\mu\mathrm{e}^{-\mu\xi}\int_{-\infty}^{\xi}\mathrm{e}^{\mu \zeta}\mathscr{H}^\mu[\overline{U}](\xi)\mathrm{d}\zeta\\
			=&\mu \mathrm{e}^{-\mu\xi}\int_{-\infty}^{\xi}\mathrm{e}^{\mu \zeta}\left(\mathscr{H}^\mu[\overline{U}](\xi)-\mathscr{H}^\mu[\overline{U}](\zeta)\right)\mathrm{d}\zeta\\
			\geq&0. 		
		\end{aligned}	
	\end{equation*}
	Similarly, we conclude that $0\leq \underline{U}(\xi)\leq U_{n+1}(\xi)\leq U_{n}(\xi)\leq \overline{U}(\xi)\leq1$ and $U_n'(\xi)\geq0$, $\xi\in\mathbb{R}$, $n=1,2,3,\cdots$. Consequently, the limit of $U_n(\xi)$ exists when $n\rightarrow+\infty$. Let $\lim_{n\rightarrow+\infty}U_n(\xi):=U(\xi)$. Then $U(\xi)$ is non-decreasing in $\xi$ over $\mathbb{R}$. By Lebesgue's dominated convergence theorem and letting $n\rightarrow+\infty$ in \eqref{eq3.12}, we see that \eqref{eq3.12} becomes \eqref{eq3.10}, which implies that $U(\xi)$ satisfies \eqref{eq1.17}.	
	
	Subsequently, we prove that $U(-\infty)=0$ and $U(+\infty)=1$. The existence of $U(\pm\infty)$ can be demonstrated by $U'(\xi)\geq0$ and $0\leq U(\xi)\leq1$. It follows from \eqref{eq3.10} that $\overline{U}(-\infty)=0$, which, together with $0\leq U(\xi)\leq\overline{U}(\xi)$, means $U(-\infty)=0$. From \eqref{eq3.10} and L'Hospital's rule, we conclude that
	\begin{equation*}	
		\begin{aligned}		
			\lim_{\xi\rightarrow+\infty}U(\xi)=&\lim_{\xi\rightarrow+\infty}\frac{\int_{-\infty}^{\xi}\mathrm{e}^{\mu \zeta}\mathscr{H}^\mu[U](\zeta)\mathrm{d}\zeta}{\mathrm{e}^{\mu\xi}}\\
			=&\lim_{\xi\rightarrow+\infty}\frac{\mu U(\xi)+\frac{1}{6c}\mathcal{D}_h[U](\xi)+\frac{1}{c}f(U(\xi))}{\mu}.
		\end{aligned}	
	\end{equation*}
	Thus, we have $\mathcal{D}_h[U](+\infty)=0$ and $f(U(+\infty))=0$, which imply that $U(+\infty)=0$ or $1$. Since $U'(\xi)\geq0$ and $\underline{U}(\xi)\not\equiv0$, there exists $\xi^*\in\mathbb{R}$ such that $U(+\infty)\geq U(\xi^*)\geq \underline{U}(\xi^*)>0$. Therefore we get $U(+\infty)=1$.
	
	\emph{Case} $c=c^*$. We choose a sequence of solutions $\{c^n, U^n(\xi)\}_{n=1}^{+\infty}$ of \eqref{eq1.17} such that $(U^n(\xi))'\geq0$ and $c^n\rightarrow c^*$ as $n\rightarrow+\infty$. Obviously, $\{U^n(\xi)\}$ is equicontinuous and uniformly bounded on $\mathbb{R}$. From the Arzela-Ascoli theorem, we see that there exists a subsequence $\{U^{n_k}(\xi)\}_{k=1}^{+\infty}$ of $\{U^{n}(\xi)\}_{n=1}^{+\infty}$ and $U^*(\xi)\in[0,1]$ such that $\lim_{k\rightarrow+\infty}U^{n_k}(\xi)=U^*(\xi)$ uniformly on any compact subset of $\mathbb{R}$. Since $(c^{n_k}, U^{n_k}(\xi))$ is a solution of \eqref{eq1.17}, by taking $k\rightarrow+\infty$, we obtain that $U^*(\xi)$ is also a solution of \eqref{eq1.17} and $(U^*(\xi))'\geq0$ for all $\xi\in\mathbb{R}$.
	
	Subsequently, we prove that $U^*(-\infty)=0$ and $U^*(+\infty)=1$. Note that $\int_{-\infty}^{+\infty}\mathcal{D}_h[U](\xi)\mathrm{d}\xi$ $=0$. By integrating the first equation of \eqref{eq1.17} with solution $(c^{n_k}, U^{n_k})$ from $-\infty$ to $+\infty$, we obtain
	\begin{equation}
		\label{eq3.14}	
		c^{n_k}=\int_{-\infty}^{+\infty}f(U^{n_k}(\xi))\mathrm{d}\xi.
	\end{equation}
	Here, by translating $U^{n_k}(\xi)$, we can suppose $U^{n_k}(0)=\frac{1}{2}$ for any $k$. Since $(U^*(\xi))'\geq0$ and $0\leq U^*(\xi)\leq1$, then $U^*(\pm\infty)$ exists and $0\leq U^*(\pm\infty)\leq1$. For \eqref{eq3.14}, taking $k\rightarrow+\infty$ and applying Fatous' lemma, we have 
	\begin{equation}
		\label{eq3.15}
		\int_{-\infty}^{+\infty}f(U^{*}(\xi))\mathrm{d}\xi=\int_{-\infty}^{+\infty}\liminf_{k\rightarrow+\infty}f(U^{n_k}(\xi))\mathrm{d}\xi\leq\liminf_{k\rightarrow+\infty}\int_{-\infty}^{+\infty}f(U^{n_k}(\xi))\mathrm{d}\xi=c^*.
	\end{equation}
	Obviously, $f(U^*(\pm\infty))\geq0$. If $f(U^*(\pm\infty))>0$, then $\int_{-\infty}^{+\infty}f(U^{*}(\pm\infty))\mathrm{d}\xi=+\infty$, which contradicts with \eqref{eq3.15}. Thus, we have $f(U^*(\pm\infty))=0$, and therefore $U^*(\pm\infty)=0$ or $1$. Since $U^*(0)=\frac{1}{2}$, it follows that $U^*(-\infty)=0$ and $U^*(+\infty)=1$.
	
	\textbf{The necessity}. By Lemma \ref{le3.3}, we prove the necessity. Assume that $(c,U(\xi))$ is a solution of \eqref{eq1.17}. Since $U(\xi)\rightarrow0$ as $\xi\rightarrow-\infty$, we have $\lim_{\xi\rightarrow-\infty}\frac{f(U(\xi))}{U(\xi)}=f'(0)$. That is, for any $\varepsilon>0$, there exists $\xi_1<0$ such that
	\begin{equation}
		\label{eq3.16}
		\frac{f(U(\xi))}{U(\xi)}\geq f'(0)-\varepsilon, \ \xi<\xi_1.
	\end{equation}
	
	Now, we establish a key inequality. Recall that $\delta_m$ $(m=1,2,3)$ defined in \eqref{eq3.3}. Without loss of generality, let $\delta_m\geq0$ and $\xi_1<\xi_2$, for any $\xi\in[\xi_1,\xi_2]$, applying Jensen's inequality, we obtain
	\begin{equation}
		\label{eq3.17}
		\begin{aligned}	
			\frac{1}{\xi_2-\xi_1}\int_{\xi_1}^{\xi_2}\frac{U(\xi+\delta_m)}{U(\xi)}\mathrm{d}\xi=&\frac{1}{\xi_2-\xi_1}\int_{\xi_1}^{\xi_2}\mathrm{e}^{\ln U(\xi+\delta_m)-\ln U(\xi)}\mathrm{d}\xi\\
			\geq&\mathrm{e}^{\frac{1}{\xi_2-\xi_1}\int_{\xi_1}^{\xi_2}\left(\ln U(\xi+\delta_m)-\ln U(\xi)\right)\mathrm{d}\xi}\\
			=&\mathrm{e}^{\frac{1}{\xi_2-\xi_1}\left(\int_{\xi_1+\delta_m}^{\xi_2+\delta_m}\ln U(\xi)\mathrm{d}\xi-\int_{\xi_1}^{\xi_2}\ln U(\xi)\mathrm{d}\xi\right)}\\
			=&\mathrm{e}^{\frac{1}{\xi_2-\xi_1}\int_{\xi_1+\delta_m}^{\xi_1}\ln U(\xi)\mathrm{d}\xi+\frac{1}{\xi_2-\xi_1}\int_{\xi_2}^{\xi_2+\delta_m}\ln U(\xi)\mathrm{d}\xi}\\
			=&\mathrm{e}^{\frac{\delta_m\left(\ln U(\xi_2)-\ln U(\xi_1)\right)}{\xi_2-\xi_1}+\frac{1}{\xi_2-\xi_1}\left(\int_{\xi_1+\delta_m}^{\xi_1}\ln\frac{U(\xi)}{U(\xi_1)}\mathrm{d}\xi+\int_{\xi_2}^{\xi_2+\delta_m}\ln \frac{U(\xi)}{U(\xi_2)}\mathrm{d}\xi\right)}.
		\end{aligned}		
	\end{equation}
	Similarly, for $\delta_m\geq0$, we have
	\begin{equation}
		\label{eq3.18}
		\begin{aligned}	
			\frac{1}{\xi_2-\xi_1}\int_{\xi_1}^{\xi_2}\frac{U(\xi-\delta_m)}{U(\xi)}\mathrm{d}\xi\geq\mathrm{e}^{\frac{\delta_m\left(\ln U(\xi_1)-\ln U(\xi_2)\right)}{\xi_2-\xi_1}+\frac{1}{\xi_2-\xi_1}\left(\int_{\xi_1-\delta_m}^{\xi_1}\ln\frac{U(\xi)}{U(\xi_1)}\mathrm{d}\xi+\int_{\xi_2}^{\xi_2-\delta_m}\ln \frac{U(\xi)}{U(\xi_2)}\mathrm{d}\xi\right)}.
		\end{aligned}		
	\end{equation}
	For the sake of further calculations, we set
	\begin{equation}
		\label{eq3.20}
		\begin{aligned}	
			&\lambda(\xi_1,\xi_2):=\frac{\ln U(\xi_2)-\ln U(\xi_1)}{\xi_2-\xi_1},\\
			&\varepsilon(\delta_m):=\frac{1}{\xi_2-\xi_1}\left(\int_{\xi_2}^{\xi_2+\delta_m}\ln\frac{U(\xi)}{U(\xi_2)}\mathrm{d}\xi-\int_{\xi_1}^{\xi_1+\delta_m}\ln \frac{U(\xi)}{U(\xi_1)}\mathrm{d}\xi\right),\ m=1,2,3,\\
			&\varepsilon(-\delta_m):=\frac{1}{\xi_2-\xi_1}\left(\int_{\xi_1-\delta_m}^{\xi_1}\ln\frac{U(\xi)}{U(\xi_1)}\mathrm{d}\xi-\int_{\xi_2-\delta_m}^{\xi_2}\ln\frac{U(\xi)}{U(\xi_2)}\mathrm{d}\xi\right), \ m=1,2,3.
		\end{aligned}	
	\end{equation}
	
	Next, we prove $c\geq c^*$. From the first equation of \eqref{eq1.17}, we obtain
	\begin{equation*}
		c\int_{\xi_1}^{\xi_2}\frac{U'(\xi)}{U(\xi)}\mathrm{d}\xi=\frac{1}{6}\int_{\xi_1}^{\xi_2}\frac{\mathcal{D}_h[U](\xi)}{U(\xi)}\mathrm{d}\xi+\int_{\xi_1}^{\xi_2}\frac{f(U(\xi))}{U(\xi)}\mathrm{d}\xi,
	\end{equation*}
	which, combining with \eqref{eq3.16}, implies  
	\begin{equation}
		\label{eq3.19}
		c(\ln U(\xi_2)-\ln U(\xi_1))\geq \frac{1}{6}\int_{\xi_1}^{\xi_2}\frac{\mathcal{D}_h[U](\xi)}{U(\xi)}\mathrm{d}\xi+(f'(0)-\varepsilon)(\xi_2-\xi_1).
	\end{equation}
	By \eqref{eq3.17}, \eqref{eq3.18} and \eqref{eq3.20}, it follows that \eqref{eq3.19} becomes
	\begin{equation}
		\label{eq3.21}
		\begin{aligned}	
			c\lambda(\xi_1,\xi_2)\geq& \frac{1}{6(\xi_2-\xi_1)}\int_{\xi_1}^{\xi_2}\frac{\mathcal{D}_h[U](\xi)}{U(\xi)}\mathrm{d}\xi+f'(0)-\varepsilon\\
			=&\frac{1}{6(\xi_2-\xi_1)}\sum_{m=1}^{3}\left(\int_{\xi_1}^{\xi_2}\ln\frac{U(\xi+\delta_m)}{U(\xi)}\mathrm{d}\xi+\int_{\xi_1}^{\xi_2}\ln\frac{U(\xi-\delta_m)}{U(\xi)}\mathrm{d}\xi\right)-1+f'(0)-\varepsilon\\ 
			\geq&\frac{1}{6}\sum_{m=1}^{3}\left(\mathrm{e}^{\lambda(\xi_1,\xi_2)\delta_m+\varepsilon(\delta_m)}+\mathrm{e}^{-\lambda(\xi_1,\xi_2)\delta_m+\varepsilon(-\delta_m)}\right)-1+f'(0)-\varepsilon.
		\end{aligned}		
	\end{equation}
	From Lemma \ref{le3.3}, we conclude that $(\xi_2-\xi_1)\varepsilon(\pm\delta_m)$ is bounded with respect to $\xi_2$ for $m=1,2,3$. Hence, we can choose a sufficiently small $\xi_1$ such that $\xi_2-\xi_1\rightarrow+\infty$ and $U(\xi_1)<U(\xi_2)$, which implies that $|\varepsilon(\pm\delta_m)|<\varepsilon$ for $m=1,2,3$ and $\lambda(\xi_1,\xi_2)>0$. By \eqref{eq3.21}, it follows that
	\begin{equation*}
		\begin{aligned}	
			c\geq&\frac{1}{\lambda(\xi_1,\xi_2)}\left[\frac{1}{6}\sum_{m=1}^{3}\left(\mathrm{e}^{\lambda(\xi_1,\xi_2)\delta_m-\varepsilon}+\mathrm{e}^{-\lambda(\xi_1,\xi_2)\delta_m-\varepsilon}\right)-1+f'(0)-\varepsilon\right]\\
			\geq&\inf_{\lambda>0}\frac{1}{\lambda}\left[\frac{1}{6}\sum_{m=1}^{3}\left(\mathrm{e}^{\lambda\delta_m-\varepsilon}+\mathrm{e}^{-\lambda\delta_m-\varepsilon}\right)-1+f'(0)-\varepsilon\right].
		\end{aligned}		
	\end{equation*}	
	Moreover, taking $\varepsilon\rightarrow0$ in the above inequality, we have $c\geq c^*$. The proof is completed.
\end{proof}

\section{The uniqueness of traveling waves}\label{se3} 

In this section we use methods in \cite{CC2, CG2, DK, GW2} to obtain a set of lemmas, proving the uniqueness of traveling waves of \eqref{eq1.17}, i.e., Theorem \ref{th1.4}. Throughout this section, we always assume that $c\geq c^*$. 

\subsection{Preliminaries}\label{se31}
According to \cite{CC2, DK}, Ikehara's theorem (see \cite[Theorem 2.12]{EE}) is an important technique in proving the uniqueness, and the following lemma is a prerequisite for the application of Ikehara's theorem.
\begin{lemma}
	\label{le3.4}
	Let $(c,U(\xi))$ be a solution of \eqref{eq1.17}. Then there exists $\omega>0$ such that $U(\xi)=O(\mathrm{e}^{\omega\xi})$ as $\xi\rightarrow-\infty$.	
\end{lemma}
\begin{proof}
	For any given $\zeta\in\mathbb{R}$, let 
	\begin{equation*}
		\mathscr{F}(\zeta):=\inf_{\xi\leq\zeta}\frac{f(U(\xi))}{U(\xi)}=\inf_{s\in[0,\varpi(\zeta)]}\frac{f(s)}{s}>0, \ \varpi(\zeta):=\sup_{\xi\leq\zeta}U(\xi)\in(0,1),\ \xi\in\mathbb{R}.
	\end{equation*}	
	Taking $\zeta_1<\zeta$ on $\mathbb{R}$, and integrating the first equation of \eqref{eq1.17} from $\zeta_1$ to $\zeta$, we have
	\begin{equation}
		\label{eq3.220}
		c(U(\zeta)-U(\zeta_1))-\int_{\zeta_1}^{\zeta}\frac{1}{6}\mathcal{D}_h[U](\xi)\mathrm{d}\xi=\int_{\zeta_1}^{\zeta}f(U(\xi))\mathrm{d}\xi\geq\mathscr{F}(\zeta)\int_{\zeta_1}^{\zeta}U(\xi)\mathrm{d}\xi.
	\end{equation} 
	Without loss of generality, let $\delta_m\geq0, m=1,2,3$. Then it follows from $U(\xi)\in[0,1]$, $\xi\in\mathbb{R}$, that 
	\begin{equation}
		\label{eq3.22}
		\begin{aligned}
			&\mathscr{F}(\zeta)\int_{\zeta_1}^{\zeta}U(\xi)\mathrm{d}\xi\\
			\leq&c(U(\zeta)-U(\zeta_1))-\int_{\zeta_1}^{\zeta}\frac{1}{6}\left[\sum_{m=1}^{3}(U(\xi+\delta_m)+U(\xi-\delta_m))-6U(\xi)\right]\mathrm{d}\xi\\
			=&c-\frac{1}{6}\sum_{m=1}^{3}\left[\int_{\zeta_1-\delta_m}^{\zeta-\delta_m}U(\xi)\mathrm{d}\xi-\int_{\zeta_1}^{\zeta}U(\xi)\mathrm{d}\xi+\int_{\zeta_1+\delta_m}^{\zeta+\delta_m}U(\xi)\mathrm{d}\xi-\int_{\zeta_1}^{\zeta}U(\xi)\mathrm{d}\xi\right]\\
			=&c-\frac{1}{6}\sum_{m=1}^{3}\left[\int_{\zeta}^{\zeta+\delta_m}U(\xi)\mathrm{d}\xi-\int_{\zeta-\delta_m}^{\zeta}U(\xi)\mathrm{d}\xi+\int_{\zeta_1-\delta_m}^{\zeta_1}U(\xi)\mathrm{d}\xi-\int_{\zeta_1}^{\zeta_1+\delta_m}U(\xi)\mathrm{d}\xi\right]\\
			\leq&c+\frac{1}{6}\sum_{m=1}^{3}2\delta_m\\
			\leq&c+1.
		\end{aligned}	
	\end{equation} 
	Note that $\mathscr{F}(-\infty)=f'(0)$ and $\mathscr{F}(\zeta)$ is non-increasing. By taking  $\zeta_1\rightarrow-\infty$ in \eqref{eq3.22}, we obtain
	\begin{equation*}
		\int_{-\infty}^{\zeta}U(\xi)\mathrm{d}\xi\leq\frac{c+1}{\mathscr{F}(\zeta)}\leq\frac{c+1}{\mathscr{F}(0)},
	\end{equation*}   
	which implies that $U(\xi)$ is integrable on $(-\infty,\zeta)$ for $\zeta\in\mathbb{R}$. Let 
	\begin{equation}
		\label{eq3.23}
		V(\zeta):=\int_{-\infty}^{\zeta}U(\xi)\mathrm{d}\xi.
	\end{equation}	
	Then $V(\zeta)$ is increasing with respect to $\zeta$ and $V(\zeta)\in(0,+\infty)$. By choosing $\zeta_2\leq0$, we deduce from \eqref{eq3.220} that
	\begin{equation*}
		cV(\zeta_2)\geq\int_{-\infty}^{\zeta_2}\frac{1}{6}\mathcal{D}_h[V](\zeta)\mathrm{d}\zeta+\mathscr{F}(0)\int_{-\infty}^{\zeta_2}V(\zeta)\mathrm{d}\zeta,
	\end{equation*}  
	where 
	\begin{equation*}
		\begin{aligned}
			\int_{-\infty}^{\zeta_2}\frac{1}{6}\mathcal{D}_h[V](\zeta)\mathrm{d}\zeta=&\frac{1}{6}\int_{-\infty}^{\zeta_2}\left[\sum_{m=1}^{3}(V(\zeta+\delta_m)+V(\zeta-\delta_m)-2V(\zeta))\right]\mathrm{d}\zeta\\
			=&\frac{1}{6}\sum_{m=1}^{3}\left(\int_{\zeta_2}^{\zeta_2+\delta_m}V(\zeta)\mathrm{d}\zeta-\int_{\zeta_2-\delta_m}^{\zeta_2}V(\zeta)\mathrm{d}\zeta\right)\\
			\geq&0.
		\end{aligned}
	\end{equation*} 
	It follows that there exists $\zeta_3>0$ such that  
	\begin{equation}
		\label{eq3.26}
		cV(\zeta_2)\geq \mathscr{F}(0)\int_{\zeta_2-\zeta_3}^{\zeta_2}V(\zeta)\mathrm{d}\zeta\geq \mathscr{F}(0)\zeta_3V(\zeta_2-\zeta_3).
	\end{equation}   
	We choose $\zeta_3>\frac{c}{\mathscr{F}(0)}$, then there exists $\omega=\frac{1}{\zeta_3}\ln\frac{\zeta_3\mathscr{F}(0)}{c}>0$ such that $\mathrm{e}^{-\omega\zeta_3}=\frac{c}{\zeta_3\mathscr{F}(0)}$. Multiplying both sides of \eqref{eq3.26} by $\frac{\mathrm{e}^{-\omega\zeta_2}}{c}$, we obtain  
	\begin{equation*}
		\mathrm{e}^{-\omega\zeta_2}V(\zeta_2)\geq\mathrm{e}^{-\omega\zeta_2}\frac{\zeta_3\mathscr{F}(0)}{c}V(\zeta_2-\zeta_3)=\mathrm{e}^{-\omega(\zeta_2-\zeta_3)}V(\zeta_2-\zeta_3), \ \zeta_2\leq0.
	\end{equation*}   
	Let $\overline{V}:=\max_{\zeta_2\in[-\zeta_3,0]}\{\mathrm{e}^{-\omega\zeta_2}V(\xi_2)\}$. Then 
	\begin{equation}
		\label{eq3.27}	
		\mathrm{e}^{-\omega\zeta}V(\zeta)\leq\overline{V}, \ \zeta\in(-\infty,0],
	\end{equation}
	which implies that  $V(\zeta)=O(\mathrm{e}^{\omega\zeta})$, as $\zeta\rightarrow-\infty$. By taking $\zeta_1\rightarrow-\infty$ in \eqref{eq3.220}, and combining with \eqref{eq3.23}, we have 
	\begin{equation*}
		\begin{aligned}
			cU(\zeta)=&\int_{-\infty}^{\zeta}\frac{1}{6}\mathcal{D}_h[U](\xi)\mathrm{d}\xi+\int_{-\infty}^{\zeta}f(U(\xi))\mathrm{d}\xi\\
			\leq&\frac{1}{6}\left[\sum_{m=1}^{3}(V(\zeta+\delta_m)+V(\zeta-\delta_m))-6V(\zeta)\right]+f'(0)V(\zeta).
		\end{aligned}
	\end{equation*} 
	Moreover, by multiplying both sides of the above inequality with $\frac{\mathrm{e}^{-\omega\zeta}}{c}$, and combining with \eqref{eq3.27}, we obtain
	\begin{equation*}
		\begin{aligned}
			\mathrm{e}^{-\omega\zeta}U(\zeta)\leq&\frac{1}{6c}\mathrm{e}^{-\omega\zeta}\sum_{m=1}^{3}(V(\zeta+\delta_m)+V(\zeta-\delta_m))+\frac{f'(0)-1}{c}\mathrm{e}^{-\omega\zeta}V(\zeta)\\
			\leq&\frac{f'(0)}{c}\overline{V}.
		\end{aligned}
	\end{equation*}
	Thus, for $\omega=\frac{1}{\zeta_3}\ln\frac{\zeta_3\mathscr{F}(0)}{c}>0$, we see that $U(\xi)=O(\mathrm{e}^{\omega\xi})$ as $\xi\rightarrow-\infty$. This completes the proof.
\end{proof}

By Lemma \ref{le3.4} and Ikehara's theorem, we obtain the following lemma about the asymptotic behavior and strict monotonicity of $U(\xi)$ for all $\xi\in\mathbb{R}$. 
\begin{lemma}
	\label{le3.42}
	Let $(c,U(\xi))$ be a solution of \eqref{eq1.17}. Then 
	\begin{equation}
		\label{eq3.290}
		\lim_{\xi\rightarrow-\infty}\frac{U'(\xi)}{U(\xi)}=\lambda_1, \ \lim_{\xi\rightarrow+\infty}\frac{U'(\xi)}{U(\xi)-1}=\lambda_0, \ \textup{for} \  c\geq c^*,
	\end{equation}
	where $\lambda_1$ is given in Proposition \ref{pr3.1}, and $\lambda_0$ is the unique negative real root of the following problem
	\begin{equation}
		\label{eq3.291}
		c\lambda-\frac{1}{6}\left[\sum_{m=1}^{3}(\mathrm{e}^{\lambda\delta_m}+\mathrm{e}^{-\lambda\delta_m})-6\right]-f'(1)=0.
	\end{equation}	
	Moreover, $U'(\xi)>0$ for $\xi\in\mathbb{R}$.
\end{lemma}
\begin{proof}
	We first show that the left limit in \eqref{eq3.290} holds. Define the two-sided Laplace transform of $U$ by
	\begin{equation*}
		L(\lambda):=\int_{\mathbb{R}}\mathrm{e}^{-\lambda\xi}U(\xi)\mathrm{d}\xi, \ \lambda\in\mathbb{C}, \ 0<\mathrm{Re}\lambda<\omega,
	\end{equation*}
	where $\omega$ is given in Lemma \ref{le3.4}. For $\lambda\in\mathbb{C}$, the equation $G(c,\lambda)$ in Proposition \ref{pr3.1} can be rewritten into
	\begin{equation*}
		G(c,\lambda)=c\lambda-\frac{1}{6}\left[\sum_{m=1}^{3}(\mathrm{e}^{\lambda\delta_m}+\mathrm{e}^{-\lambda\delta_m})-6\right]-f'(0).
	\end{equation*}
	Thus we have
	\begin{equation}	
		\label{eq3.292}
		\begin{aligned}
			\int_{\mathbb{R}}\frac{1}{6}\mathrm{e}^{-\lambda\xi}\mathcal{D}_h[U](\xi)\mathrm{d}\xi=&\frac{1}{6}\int_{\mathbb{R}}\mathrm{e}^{-\lambda\xi}\left[\sum_{m=1}^{3}(U(\xi+\delta_m)+U(\xi-\delta_m))-6U(\xi)\right]\mathrm{d}\xi\\
			=&\frac{1}{6}\int_{\mathbb{R}}\mathrm{e}^{-\lambda\xi}U(\xi)\left[\sum_{m=1}^{3}(\mathrm{e}^{\lambda\delta_m}+\mathrm{e}^{-\lambda\delta_m})-6\right]\mathrm{d}\xi\\
			=&L(\lambda)(c\lambda-G(c,\lambda)-f'(0)).
		\end{aligned}		
	\end{equation}
	Integrating the first equation of \eqref{eq1.17} multiplied by $\mathrm{e}^{-\lambda\xi}$ over $\mathbb{R}$, and combining with \eqref{eq3.292}, we obtain
	\begin{equation}	
		\label{eq3.293}
		\begin{aligned}
			&\int_{\mathbb{R}}\mathrm{e}^{-\lambda\xi}\left(f(U(\xi))-U(\xi)f'(0)\right)\mathrm{d}\xi\\
			=&c\int_{\mathbb{R}}\mathrm{e}^{-\lambda\xi}U'(\xi)\mathrm{d}\xi-\int_{\mathbb{R}}\frac{1}{6}\mathrm{e}^{-\lambda\xi}\mathcal{D}_h[U](\xi)\mathrm{d}\xi-\int_{\mathbb{R}}\mathrm{e}^{-\lambda\xi}U(\xi)f'(0)\mathrm{d}\xi\\
			=&c\lambda\int_{\mathbb{R}}\mathrm{e}^{-\lambda\xi}U(\xi)\mathrm{d}\xi-c\lambda L(\lambda)+G(c,\lambda)L(\lambda)\\
			=&G(c,\lambda)L(\lambda).	
		\end{aligned}			
	\end{equation}
	Since $U(-\infty)=0$, we can conclude from \eqref{eq3.2} that
	\begin{equation*}
		f'(0)U-f(U)=O(U^{1+\theta}), \ \xi\rightarrow-\infty.
	\end{equation*}
	It follows that $\int_{\mathbb{R}}\mathrm{e}^{-\lambda\xi}\left(f(U(\xi))-U(\xi)f'(0)\right)\mathrm{d}\xi$ is analytic for $0<\mathrm{Re}\lambda<\omega(1+\theta)$. If $\lambda$ is not a solution of $G(c,\lambda)=0$, then we see from \eqref{eq3.293} that $L(\lambda)$ is analytic for $0<\mathrm{Re}\lambda<\omega$. Since $U>0$, it follows from \cite[Lemma 5b]{W} that there exists a real number $B$ such that $L(\lambda)$ is analytic for $0<\mathrm{Re}\lambda< B$ and has a singularity at $\lambda=B$. By Proposition \ref{pr3.1}, we know that $L(\lambda)$ is analytic for $0<\mathrm{Re}\lambda<\lambda_1$.
	
	We rewrite \eqref{eq3.293} as 
	\begin{equation}
		\label{eq3.294}
		\int_{-\infty}^{0}\mathrm{e}^{-\lambda\xi}U(\xi)\mathrm{d}\xi=\frac{\int_{\mathbb{R}}\mathrm{e}^{-\lambda\xi}\left(f(U(\xi))-U(\xi)f'(0)\right)\mathrm{d}\xi}{G(c,\lambda)}-\int_{0}^{+\infty}\mathrm{e}^{-\lambda\xi}U(\xi)\mathrm{d}\xi.
	\end{equation}
	Assume that $G(c,\lambda)=0$ has any complex form solution $\lambda$ with $\mathrm{Re}\lambda=\lambda_1$ and $\mathrm{Im}\lambda=\lambda_3$. Then we have
	\begin{equation*}
		\left\{\begin{aligned}
			&c\lambda_1=\frac{1}{6}\sum_{m=1}^{3}\left(\mathrm{e}^{\lambda_1\delta_m}+\mathrm{e}^{-\lambda_1\delta_m}\right)\cos(\lambda_3\delta_m)-1+f'(0),\\
			&c\lambda_3=\frac{1}{6}\sum_{m=1}^{3}\left(\mathrm{e}^{\lambda_1\delta_m}+\mathrm{e}^{-\lambda_1\delta_m}\right)\sin(\lambda_3\delta_m).
		\end{aligned}\right.
	\end{equation*}
	From the above equations and $G(c,\lambda_1)=0$, we deduce that $\cos(\lambda_3\delta_m)=1$ and $\sin(\lambda_3\delta_m)=0$. By $c>0$, we obtain $\lambda_3=0$. This implies $\lambda=\lambda_1\in\mathbb{R}$. Therefore, $G(c,\lambda)=0$ does not have any solution with $\mathrm{Re}\lambda=\lambda_1$ except for $\lambda=\lambda_1$.
	
	Suppose that $U(\xi)$ is non-decreasing for small $\xi<0$. By \eqref{eq3.294}, we define
	\begin{equation*}
		H(\lambda):=\frac{(\lambda-\lambda_1)^{k+1}\int_{\mathbb{R}}\mathrm{e}^{-\lambda\xi}\left(f(U(\xi))-U(\xi)f'(0)\right)\mathrm{d}\xi}{G(c,\lambda)}-(\lambda-\lambda_1)^{k+1}\int_{0}^{+\infty}\mathrm{e}^{-\lambda\xi}U(\xi)\mathrm{d}\xi,
	\end{equation*}
	where $k=0$ if $c>c^*$, $k=1$ if $c=c^*$. Since $\int_{0}^{+\infty}\mathrm{e}^{-\lambda\xi}U(\xi)\mathrm{d}\xi$ is analytic on the strip $0<\mathrm{Re}\lambda\leq\lambda_1$, it follows from Proposition \ref{pr3.1} that $\lim_{\lambda\rightarrow\lambda_1}H(\lambda)$ exists, which implies that $H(\lambda)$ is analytic on the strip $0<\mathrm{Re}\lambda\leq\lambda_1$. From Ikehara's theorem given in \cite[Theorem 2.12]{EE}, we can conclude that $\lim_{\xi\rightarrow-\infty}\frac{U(\xi)}{\mathrm{e}^{\lambda_1\xi}}$ exists, and  $\lim_{\xi\rightarrow-\infty}\frac{U(\xi)}{|\xi|\mathrm{e}^{\lambda_1\xi}}$ exists. Then we set
	\begin{equation*}
		\lim_{\xi\rightarrow-\infty}\frac{U(\xi)}{\mathrm{e}^{\lambda_1\xi}}:=r_1\in[0,+\infty), \ 	\lim_{\xi\rightarrow-\infty}\frac{U(\xi)}{|\xi|\mathrm{e}^{\lambda_1\xi}}:=r_2\in[0,+\infty).
	\end{equation*}
	By Lemma \ref{le3.4} and $\lambda_1<\omega$, we deduce that $r_1\neq0$ and $r_2\neq0$. Therefore there exist $\vartheta_1=\frac{1}{\lambda_1}\ln\frac{1}{r_1}$ and $\vartheta_2=\frac{1}{\lambda_1}\ln\frac{1}{r_2}$ such that
	\begin{equation}
		\label{eq3.295}
		\lim_{\xi\rightarrow-\infty}\frac{U(\xi+\vartheta_1)}{\mathrm{e}^{\lambda_1\xi}}=1 \ \textup{for} \  c>c^*, \ 	\lim_{\xi\rightarrow-\infty}\frac{U(\xi+\vartheta_2)}{|\xi|\mathrm{e}^{\lambda_1\xi}}=1 \ \textup{for} \  c=c^*.
	\end{equation}	
	
	Suppose that $U(\xi)$ is non-monotone for small $\xi<0$. Then we set $\hat{U}(\xi):=U(\xi)\mathrm{e}^{\frac{\xi}{c}}$. Since 
	\begin{equation*}
		cU'(\xi)=\frac{1}{6}\left[\sum_{m=1}^{3}(U(\xi+\delta_m)+U(\xi-\delta_m))-6U(\xi)\right]+f(U(\xi))\geq-U(\xi),
	\end{equation*}
	it follows that
	\begin{equation*}
		\frac{\mathrm{d}\hat{U}(\xi)}{\mathrm{d}\xi}=\frac{1}{c}\left(cU'(\xi)+U(\xi)\right)\mathrm{e}^{\frac{\xi}{c}}\geq0,
	\end{equation*}
	which implies that $\hat{U}(\xi)$ is non-decreasing. We further define the two-sided Laplace transform of $\hat{U}$ by
	\begin{equation*}
		\hat{L}(\lambda):=\int_{\mathbb{R}}\mathrm{e}^{-\lambda\xi}\hat{U}(\xi)\mathrm{d}\xi, \ \lambda\in\mathbb{C}, \ 0<\mathrm{Re}\lambda<\omega.
	\end{equation*}
	Clearly, $\hat{L}(\lambda)=L(\lambda-\frac{1}{c})$. By repeating the above argument, we can deduce that there exist $\hat{\vartheta}_1$ and $\hat{\vartheta}_2$ such that 
	\begin{equation}
		\label{eq3.296}
		\lim_{\xi\rightarrow-\infty}\frac{\hat{U}(\xi+\hat{\vartheta}_1)}{\mathrm{e}^{(\lambda_1-\frac{1}{c})\xi}}=1 \ \textup{for} \  c>c^*, \  \text{ and } \	\lim_{\xi\rightarrow-\infty}\frac{\hat{U}(\xi+\hat{\vartheta}_2)}{|\xi|\mathrm{e}^{(\lambda_1-\frac{1}{c})\xi}}=1 \ \textup{for} \  c=c^*.
	\end{equation}	
	From \eqref{eq3.295} and \eqref{eq3.296}, we can conclude that $\lim_{\xi\rightarrow-\infty}\frac{U'(\xi)}{U(\xi)}=\lambda_1$ for  $c\geq c^*$.
	
	Next, we set $\mathcal{U}:=1-U$ and $\mathcal{F}(s):=f(1-s)$, it follows from \eqref{eq1.17} that
	\begin{equation*}
		c\mathcal{U}'(\xi)=\frac{1}{6}\mathcal{D}_h[\mathcal{U}](\xi)-\mathcal{F}(\mathcal{U}(\xi)), \ \xi\in\mathbb{R}.
	\end{equation*}
	It is evident that $\mathcal{F}'(0)=-f'(1)$ and $\mathcal{U}$ is non-increasing. From the methods used to prove $\lim_{\xi\rightarrow-\infty}\frac{U'(\xi)}{U(\xi)}=\lambda_1$, we conclude that $\lim_{\xi\rightarrow+\infty}\frac{U'(\xi)}{U(\xi)-1}=\lambda_0$, where $\lambda_0$ is the unique negative real root of \eqref{eq3.291}.
	
	Finally, we prove $U'(\xi)>0$ for $\xi\in\mathbb{R}$. From \eqref{eq3.290}, we can deduce that there exist $\underline{\xi}\ll0$ and $\overline{\xi}\gg0$ such that $U'(\xi)>0$ for $\xi\in(-\infty, \underline{\xi}]\cup[\overline{\xi},+\infty)$. For the solution $(c,U(\xi))$ of \eqref{eq1.17}, we can see from the proof of the sufficiency that $U'(\xi)\geq0$ for $\xi\in\mathbb{R}$, which, combining with $U(-\infty)=0$ and $U(+\infty)=1$, implies $U(\xi)\in(0,1)$ for $\xi\in[\underline{\xi},\overline{\xi}]$. Subsequently, we give the proof for $U'(\xi)>0$ for $\xi\in[\underline{\xi},\overline{\xi}]$. Assume that there exists $\hat{\xi}\in[\underline{\xi},\overline{\xi}]$ such that $U'(\hat{\xi})=0$. Recalling the equivalence form \eqref{eq3.10} of \eqref{eq1.17}, and taking the derivative of the two sides of \eqref{eq3.10} with respect to $\xi$, we obtain
	\begin{equation*}
		\begin{aligned}	
			U'(\xi)=&-\mu\mathrm{e}^{-\mu\xi}\int_{-\infty}^{\xi}\mathrm{e}^{\mu\zeta}\mathscr{H}^\mu[U](\zeta)\mathrm{d}\zeta+\mathscr{H}^\mu[U](\xi)\\
			=&-\mu\mathrm{e}^{-\mu\xi}\int_{-\infty}^{\xi}\mathrm{e}^{\mu\zeta}\left(\mathscr{H}^\mu[U](\zeta)-\mathscr{H}^\mu[U](\xi)\right)\mathrm{d}\zeta.
		\end{aligned}			
	\end{equation*}
	Since $U'(\xi)\geq0$ and $U'(\hat{\xi})=0$, it follows that
	\begin{equation*}
		0=-\mu\mathrm{e}^{-\mu\hat{\xi}}\int_{-\infty}^{\hat{\xi}}\mathrm{e}^{\mu\zeta}\left(\mathscr{H}^\mu[U](\zeta)-\mathscr{H}^\mu[U](\hat{\xi})\right)\mathrm{d}\zeta\geq0.		
	\end{equation*}
	Hence, we have $\mathscr{H}^\mu[U](\zeta)=\mathscr{H}^\mu[U](\hat{\xi})$ for $\zeta\leq\hat{\xi}$. By replacing $\xi$ of \eqref{eq3.8} with $\zeta$, and sending $\zeta\rightarrow-\infty$, we have $\mathscr{H}^\mu[U](\hat{\xi})=0$. From the first equation of \eqref{eq1.17}, we see that $\frac{1}{6}\mathcal{D}_h[U](\hat{\xi})+f(U(\hat{\xi}))=0$. It follows from \eqref{eq3.8} that $\mu U(\hat{\xi})=0$, which is  contradicting to $0<U(\hat{\xi})<1$. This implies $U'(\xi)>0$ for all $\xi\in\mathbb{R}$, which completes the proof.
\end{proof}

Based on Lemma~\ref{le3.42}, we follow the method in \cite{CG2} to give the following key lemma. Before doing so, for a fixed solution $(c,U(\xi))$ of \eqref{eq1.17}, we define 
\begin{equation}
	\label{eq3.30}
	\rho:=\rho(U)=\inf\left\{\frac{U(\xi)}{-U'(\xi)}\bigg| U(\xi)\leq\tau_0\right\}.
\end{equation}
From $U'(\xi)>0$ for $\xi\in\mathbb{R}$ and $\lim_{\xi\rightarrow-\infty}\frac{U'(\xi)}{U(\xi)}=\lambda_1$ in Lemma~\ref{le3.42}, we can conclude that $\rho\in(-\infty,0)$.
\begin{lemma}
	\label{le3.5}
	Suppose $(c,U_1(\xi))$ and $(c,U_2(\xi))$ are solutions of \eqref{eq1.17}, and there exists $\tau\in(0,\tau_0]$ such that $(1+\tau)U_2(\xi+\rho\tau)\geq U_1(\xi)$ for all $\xi\in\mathbb{R}$, where $\rho=\rho(U_2)$. Then $U_2(\xi)\geq U_1(\xi)$ for all $\xi\in\mathbb{R}$.
\end{lemma}
\begin{proof}
	Let $\varphi(\tau,\xi):=(1+\tau)U_2(\xi+\rho\tau)-U_1(\xi)$ and $\tau^*:=\inf\{\tau>0|\varphi(\tau,\xi)\geq0, \xi\in\mathbb{R}\}$. Then we obtain $\varphi(\tau^*,\xi)\geq0$ for all $\xi\in\mathbb{R}$ by the continuity of $\varphi$.
	
	Suppose, contrary to $\tau^*=0$, that $\tau^*\in(0,\tau_0]$. When $\xi\in\{\zeta|0\leq U_2(\zeta+\rho\tau)\leq\tau_0\}$, we deduce from \eqref{eq3.30} that 
	\begin{equation*}
		\frac{\partial\varphi(\tau,\xi)}{\partial\tau}=U_2(\xi+\rho\tau)+\rho(1+\tau)U'_2(\xi+\rho\tau)<U_2(\xi+\rho\tau)+\rho U'_2(\xi+\rho\tau)\leq0.
	\end{equation*}
	Since $\varphi(\tau^*,+\infty)=\tau^*$, it follows that there exists $\xi_0\in\{\zeta|\tau_0<U_2(\zeta+\rho\tau^*)\leq1\}$ such that
	\begin{equation*}
		\frac{\partial\varphi(\tau,\xi)}{\partial\xi}\bigg|_{(\tau^*,\xi_0)}=\varphi(\tau^*,\xi_0)=0,\ \varphi(\tau^*,\xi_0\pm\delta_m)\geq0,
	\end{equation*}	
	where $\delta_m$ $(m=1,2,3)$ is defined in \eqref{eq3.3}. Therefore, we get
	\begin{equation*}
		\begin{aligned}	
			&U'_1(\xi_0)=(1+\tau^*)U'_2(\xi_0+\rho\tau^*), \ U_1(\xi_0)=(1+\tau^*)U_2(\xi_0+\rho\tau^*), \\ &U_1(\xi_0\pm\delta_m)\leq(1+\tau^*)U_2(\xi_0+\rho\tau^*\pm\delta_m),
		\end{aligned}	
	\end{equation*}
	which, together with the first equation of \eqref{eq1.17}, implies that
	\begin{equation}
		\label{eq3.31}
		\begin{aligned}	
			0=&-cU'_1(\xi_0)+\frac{1}{6}\mathcal{D}_h[U_1](\xi_0)+f(U_1(\xi_0))\\
			\leq&-c(1+\tau^*)U'_2(\xi_0+\rho\tau^*)+\frac{1}{6}\mathcal{D}_h[(1+\tau^*)U_2](\xi_0+\rho\tau^*)+f((1+\tau^*)U_2(\xi_0+\rho\tau^*))\\
			=&-(1+\tau^*)\frac{1}{6}\mathcal{D}_h[U_2](\xi_0+\rho\tau^*)-(1+\tau^*)f(U_2(\xi_0+\rho\tau^*))+\frac{1}{6}\mathcal{D}_h[(1+\tau^*)U_2](\xi_0+\rho\tau^*)\\
			&+f((1+\tau^*)U_2(\xi_0+\rho\tau^*))\\
			=&f((1+\tau^*)U_2(\xi_0+\rho\tau^*))-(1+\tau^*)f(U_2(\xi_0+\rho\tau^*)).
		\end{aligned}		
	\end{equation}
	For $\tau\in(0,\tau_0]$ and $s\in(\tau_0,1]$, where $\tau_0\rightarrow1^-$, we have
	\begin{equation*}
		\frac{\partial[f((1+\tau)s)-(1+\tau)f(s)]}{\partial\tau}\bigg|_{\tau=0}=sf'(s)-f(s)<0.
	\end{equation*}
	Hence, we conclude that	
	\begin{equation}
		\label{eq3.32}
		f((1+\tau)s)-(1+\tau)f(s)<0, \tau\in(0,\tau_0], \ s\in(\tau_0,1], \  \tau_0\rightarrow1^-.
	\end{equation}
	When $\tau$ and $s$ in \eqref{eq3.32} are set to $\tau^*$ and $U_2(\xi_0+\rho\tau^*)$, respectively, it becomes evident that  \eqref{eq3.32} contradicts \eqref{eq3.31}. Consequently, we deduce that $\tau^*=0$,  which clearly implies that $U_2(\xi)\geq U_1(\xi)$ for $\xi\in\mathbb{R}$.  This completes the proof.
\end{proof} 

The following strong comparison principle will also be used to prove the uniqueness.
\begin{lemma}
	\label{le3.6}
	Let $(c,U_1(\xi))$ and $(c,U_2(\xi))$ be solutions of \eqref{eq1.17} and $U_1(\xi)\leq U_2(\xi)$ for all $\xi\in\mathbb{R}$. Then either $U_1(\xi)\equiv U_2(\xi)$ or $U_1(\xi)<U_2(\xi)$ for all $\xi\in\mathbb{R}$.
\end{lemma}
\begin{proof}
	We can assume that there exists $\xi^*\in\mathbb{R}$ such that $U_1(\xi)<U_2(\xi)$ for $\xi<\xi^*$, and $U_1(\xi^*)=U_2(\xi^*)$. It follows from \eqref{eq3.10} that
	\begin{equation*}
		U_2(\xi^*)-U_1(\xi^*)=\mathrm{e}^{-\mu\xi^*}\int_{-\infty}^{\xi^*}\mathrm{e}^{\mu\zeta}(\mathscr{H}^\mu[U_2](\zeta)-\mathscr{H}^\mu[U_1](\zeta))\mathrm{d}\zeta.
	\end{equation*}
	This implies that	
	\begin{equation}
		\label{eq3.33}
		\mathscr{H}^\mu[U_1](\xi)=\mathscr{H}^\mu[U_2](\xi), \ \xi\in(-\infty,\xi^*].
	\end{equation}	
	Without loss of generality, we let $\delta_m\geq0$, $m=1,2,3$. Then it follows from $U_1(\xi)\leq U_2(\xi)$, $\xi\in\mathbb{R}$, that 
	\begin{equation}
		\label{eq3.34}
		\begin{aligned}	
			0&\leq\sum_{m=1}^{3}(U_2(\xi+\delta_m)-U_1(\xi+\delta_m))+\sum_{m=1}^{3}(U_2(\xi-\delta_m)-U_1(\xi-\delta_m))\\
			&=\sum_{m=1}^{3}(U_2(\xi+\delta_m)+U_2(\xi-\delta_m))-\sum_{m=1}^{3}(U_1(\xi+\delta_m)+U_1(\xi-\delta_m))\\
			&=\mathcal{D}_h[U_2](\xi)+6U_2(\xi)-(\mathcal{D}_h[U_1](\xi)+6U_1(\xi)).
		\end{aligned}		
	\end{equation}	
	Meanwhile, for $\xi(-\infty,\xi^*]$, we plug \eqref{eq3.33} into \eqref{eq3.8} and obtain
	\begin{equation}
		\label{eq3.35}
		\begin{aligned}	
			&\mathcal{D}_h[U_2](\xi)+6U_2(\xi)-(\mathcal{D}_h[U_1](\xi)+6U_1(\xi))\\
			=&-6c\mu U_2(\xi)-6f(U_2(\xi))+6U_2(\xi)+6c\mu U_1(\xi)+6f(U_1(\xi))-6U_1(\xi)\\
			=&(6-6c\mu)(U_2(\xi)-U_1(\xi))-6(f(U_2(\xi))-f(U_1(\xi)))\\
			\leq&-6\left(c\mu-1-\max_{0\leq U\leq 1}|f'(U)|\right)(U_2(\xi)-U_1(\xi))\\
			\leq&0.
		\end{aligned}		
	\end{equation}
	Hence, by comparing \eqref{eq3.34} with \eqref{eq3.35}, we see that	$U_1(\xi)=U_2(\xi)$, $\xi\in(-\infty,\xi^*+\delta_m]$, $m=1,2,3$. By  replacing $\xi^*$ in the interval $(-\infty,\xi^*+\delta_m]$ with $\xi^*+\delta_m$ and repeating the above argument, we obtain $U_1(\xi)\equiv U_2(\xi)$ for $\xi\in\mathbb{R}$, which completes the proof.		
\end{proof}

\subsection{The proof of Theorem~\ref{th1.4}}\label{se32}
\begin{proof}
	Let $(c,U_1)$ and $(c,U_2)$ be solutions of \eqref{eq1.17}.  By translation if needed, we can assume that $U_1(0)=U_2(0)=\frac{1}{2}$. Theorem~\ref{th1.4} will be proved by $U_1=U_2$, the proof is divided into the following three steps.
	
	\emph{Step 1}. From \eqref{eq3.295}, we conclude that there exists some $\eta\in\mathbb{R}$ such that $\lim_{\xi\rightarrow-\infty}\frac{U_1(\xi)}{U_2(\xi)}=\mathrm{e}^{-\lambda_1\eta}$. This implies that we can suppose that $\lim_{\xi\rightarrow-\infty}\frac{U_1(\xi)}{U_2(\xi)}\leq1$ by exchanging $U_1$ and $U_2$. By $U'>0$ in Lemma \ref{le3.42}, we obtain that $\lim_{\xi\rightarrow-\infty}\frac{U_1(\xi+\zeta)}{U_2(\xi)}<1$ for $\zeta<0$. For the fixed $\zeta<0$, there exists $\xi_0<0$ such that $U_1(\xi+\zeta)<U_2(\xi)$ for $\xi\leq\xi_0$. Since $U(+\infty)=1$, for $\rho$ and $\tau_0$ defined in \eqref{eq3.30}, there exists a sufficiently small negative number $\zeta_0$ such that
	\begin{equation*}
		U_1(\xi+\zeta_0)\leq(1+\tau_0)U_2(\xi+\rho\tau_0), \ \xi\in\mathbb{R}.	
	\end{equation*}
	By Lemma~\ref{le3.5}, we obtain $U_1(\xi+\zeta_0)\leq U_2(\xi)$ for $\xi\in\mathbb{R}$. 
	
	\emph{Step 2}. Define $\zeta^*:=\sup\{\zeta<0|U_2(\xi)\geq U_1(\xi+\zeta), \xi\in\mathbb{R}\}$. To show that $\zeta^*=0$, we divide $\xi\in\mathbb{R}$ into three ranges to prove that $\zeta^*\in(-\infty,0)$ does not hold.
	
	Assume, for the sake of contradiction, that  $\zeta^*\in(-\infty,0)$. Then, from $\lim_{\xi\rightarrow-\infty}\frac{U_1(\xi+\frac{\zeta^*}{2})}{U_2(\xi+\frac{\zeta^*}{2})}\leq1$ and $U'_1>0$, we see that $\lim_{\xi\rightarrow-\infty}\frac{U_1(\xi+\zeta^*)}{U_2(\xi+\frac{\zeta^*}{2})}<1$. Hence there exists $\xi_1<0$ such that
	\begin{equation}
		\label{eq3.36}
		U_1(\xi+\zeta^*)\leq U_2(\xi+\frac{\zeta^*}{2}), \ \xi\in(-\infty,\xi_1].		
	\end{equation}
	Since $U_2(+\infty)=1$ and $U'_2(+\infty)=0$, there exists a sufficiently large $\xi_2$ such that, for $\tau\in(0,1)$ and $\rho\in(-\infty,0)$,
	\begin{equation*}
		\frac{\partial[(1+\tau)U_2(\xi+2\rho\tau)]}{\partial\tau}=U_2(\xi+2\rho\tau)+2\rho(1+\tau)U'_2(\xi+2\rho\tau)>0,	\ \xi\in[\xi_2,+\infty).
	\end{equation*}
	Therefore we get
	\begin{equation}
		\label{eq3.37}
		(1+\tau)U_2(\xi+2\rho\tau)\geq U_2(\xi)\geq U_1(\xi+\zeta^*),\ \xi\in[\xi_2,+\infty).		
	\end{equation}
	Since $U_2(\xi)\geq U_1(\xi+\zeta^*)$, $\xi\in\mathbb{R}$. It follows from Lemma~\ref{le3.6} that $U_2(\xi)>U_1(\xi+\zeta^*)$ for $\xi\in\mathbb{R}$. Furthermore, the uniform continuity of $U_2$ over $\mathbb{R}$ states that there exists $\varepsilon\in\left[0,\min\{\tau_0, \frac{\zeta^*}{4\rho}\}\right]$ such that
	\begin{equation}
		\label{eq3.38}
		(1+\varepsilon)U_2(\xi+2\rho\varepsilon)\geq U_1(\xi+\zeta^*),\ \xi\in[\xi_1,\xi_2].		
	\end{equation}
	For $\xi\in(-\infty,\xi_1]$, we see from $\varepsilon\leq\frac{\zeta^*}{4\rho}$ that $2\rho\varepsilon\geq\frac{\zeta^*}{2}$, which, together with \eqref{eq3.36}, implies that 
	\begin{equation}
		\label{eq3.39}
		(1+\varepsilon)U_2(\xi+2\rho\varepsilon)\geq U_2(\xi+\frac{\zeta^*}{2})\geq U_1(\xi+\zeta^*),\ \xi\in(-\infty,\xi_1].		
	\end{equation}
	Taking $\tau$ in \eqref{eq3.37} as $\varepsilon$, and combining \eqref{eq3.38} with \eqref{eq3.39}, we have
	\begin{equation*}
		(1+\varepsilon)U_2(\xi+2\rho\varepsilon)\geq U_1(\xi+\zeta^*),\ \xi\in\mathbb{R}.		
	\end{equation*}
	It follows from Lemma~\ref{le3.5} that $U_2(\xi+\rho\varepsilon)\geq U_1(\xi+\zeta^*)$ for $\xi\in\mathbb{R}$, contradicting to the definition of $\zeta^*$. Consequently, we obtain that $\zeta^*=0$ and $U_2(\xi)\geq U_1(\xi)$ for all $\xi\in\mathbb{R}$. 
	
	\emph{Step 3}. By applying Lemma~\ref{le3.6} with $U_1(0)=U_2(0)=\frac{1}{2}$, we know that $U_1(\xi)=U_2(\xi)$ for all $\xi\in\mathbb{R}$. The proof is completed.
\end{proof}

\section{Dependence of the minimal wave speed on an angle}\label{se4}
In this section, we prove Theorem \ref{th1.1}, which theoretically demonstrates the effect of an angle on the minimal wave speed. Additionally, we provide a graphical illustration in polar coordinates (see Figure \ref{fig2}) in Section \ref{sec:Dis}.

The proof of Theorem \ref{th1.1} can be accomplished by using the following key lemma concerning an angle $\alpha\in[0,\frac{\pi}{6}]$. Moreover, this lemma will also be illustrated in Figure \ref{fig3} in Section~\ref{sec:Dis}.

\begin{lemma}
	\label{le2.1}
	For $\alpha\in[0,\frac{\pi}{6}]$ and $n\in\mathbb{N}$, let
	\begin{equation*}
		\begin{aligned}
			\Phi_n(\alpha):=&\left(\frac{\sqrt{3}}{2}\cos\alpha+\frac{1}{2}\sin\alpha\right)\left(\frac{\sqrt{3}}{2}\sin\alpha-\frac{1}{2}\cos\alpha\right)^{2n+1}\\
			&+\left(\frac{\sqrt{3}}{2}\cos\alpha-\frac{1}{2}\sin\alpha\right)\left(\frac{\sqrt{3}}{2}\sin\alpha+\frac{1}{2}\cos\alpha\right)^{2n+1}-\sin\alpha(\cos\alpha)^{2n+1}.
		\end{aligned}
	\end{equation*}
	Then 
	\begin{equation*}	
		\Phi_n(\alpha)\left\{ \begin{aligned}
			&=0, &&\quad \mathrm{ for } \ \alpha\in[0, \ \frac{\pi}{6}], \ n=0,1,\\
			&=0, &&\quad \mathrm{ for } \  \alpha=0, \ \frac{\pi}{6}, \ n\geq2,\\
			&<0, &&\quad \mathrm{ for } \  \alpha\in(0,\frac{\pi}{6}), \ n\geq 2.
		\end{aligned}\right.
	\end{equation*}	
\end{lemma}
\begin{proof}
	Applying the polynomial expansion to $\Phi_n(\alpha)$, we have
	\begin{equation*}
		\begin{aligned}
			\Phi_n(\alpha)=&\left(\frac{\sqrt{3}}{2}\cos\alpha+\frac{1}{2}\sin\alpha\right)\sum_{l=0}^{n}\binom{2n+1}{l}\left(\frac{\sqrt{3}}{2}\sin\alpha\right)^l\left(-\frac{1}{2}\cos\alpha\right)^{2n+1-l}\\
			&+\left(\frac{\sqrt{3}}{2}\cos\alpha-\frac{1}{2}\sin\alpha\right)\sum_{l=0}^{n}\binom{2n+1}{l}\left(\frac{\sqrt{3}}{2}\sin\alpha\right)^l\left(\frac{1}{2}\cos\alpha\right)^{2n+1-l}\\
			&-\sin\alpha(\cos\alpha)^{2n+1}\\	
			=&\sum_{l=0}^{n}\left[\binom{2n+1}{2l+1}-\frac{1}{3}\binom{2n+1}{2l}\right]3\sin\alpha\left(\frac{\sqrt{3}}{2}\sin\alpha\right)^{2l}\left(\frac{1}{2}\cos\alpha\right)^{2n+1-2l}\\
			&-\sin\alpha(\cos\alpha)^{2n+1}\\
			=&\frac{3}{2}\sin\alpha\cos\alpha\sum_{l=0}^{n}\Psi(\alpha,l),
		\end{aligned}		
	\end{equation*}
	where 
	\begin{equation*}
		\Psi(\alpha,l):=\Bigg[\binom{2n+1}{2l+1}-\frac{1}{3}\binom{2n+1}{2l}\Bigg]\Bigg[\left(\frac{\sqrt{3}}{2}\sin\alpha\right)^{2l}\left(\frac{1}{2}\cos\alpha\right)^{2n-2l}-\left(\frac{1}{2}\cos\alpha\right)^{2n}\Bigg].
	\end{equation*}		
	It is evident that $\Phi_0(\alpha)=\Phi_1(\alpha)=0$ for all $\alpha\in[0,\frac{\pi}{6}]$ when $n=0$ or $1$.
	
	Since  $\sqrt{3}\sin\alpha\leq\cos\alpha$ for $\alpha\in[0,\frac{\pi}{6}]$, when $n=2$, we have
	\begin{equation*}
		\begin{aligned}
			\Phi_2(\alpha)=\frac{3}{2}\sin\alpha\cos\alpha\left[\left(\frac{\sqrt{3}}{2}\sin\alpha\right)^{2}-\left(\frac{1}{2}\cos\alpha\right)^{2}\right]\left[6\left(\frac{1}{2}\cos\alpha\right)^{2}-\frac{2}{3}\left(\frac{\sqrt{3}}{2}\sin\alpha\right)^{2}\right]\leq0. 
		\end{aligned}	
	\end{equation*}
	
	When $n=3$, using the inequality  $\sqrt{3}\sin\alpha\leq\cos\alpha$, we obtain
	\begin{equation*}
		\begin{aligned}
			\Phi_3(\alpha)=&\frac{3}{2}\sin\alpha\cos\alpha\left[\left(\frac{\sqrt{3}}{2}\sin\alpha\right)^{2}-\left(\frac{1}{2}\cos\alpha\right)^{2}\right]\Bigg[36\left(\frac{1}{2}\cos\alpha\right)^{4}-\frac{4}{3}\left(\frac{\sqrt{3}}{2}\sin\alpha\right)^{2}\\
			&+8\left(\frac{\sqrt{3}}{2}\sin\alpha\right)^{2}\left(\frac{1}{2}\cos\alpha\right)^{2}\Bigg]\leq0. 		
		\end{aligned}
	\end{equation*}	
	Clearly, $\Phi_2(\alpha), \Phi_3(\alpha)\leq0$ for all $\alpha\in[0,\frac{\pi}{6}]$, and the equal sign holds if and only if $\alpha=0$ or $\frac{\pi}{6}$.
	
	Next, we deal with cases when $n\geq4$. Since 
	\begin{equation*}	
		\binom{2n+1}{2l+1}-\frac{1}{3}\binom{2n+1}{2l}=\binom{2n+1}{2l+1}\frac{6n-8l+2}{6n-6l+3}\left\{ \begin{aligned}
			&<0, &&l>\frac{3n+1}{4},\\
			&>0, &&l<\frac{3n+1}{4},\\
			&=0, &&l=\frac{3n+1}{4},
		\end{aligned}\right.
	\end{equation*}	
	it follows that there exists $l^*=\left[\frac{3n+1}{4}\right]$ such that
	\begin{equation*}
		\label{eq2.1}
		\begin{aligned}
			\Phi_n(\alpha)=&\frac{3}{2}\sin\alpha\cos\alpha\sum_{l=0}^{2l^*-n-1}\Psi(\alpha,l)+\frac{3}{2}\sin\alpha\cos\alpha\left[\Psi(\alpha,2l^*-n)+\Psi(\alpha,l^*+1)\right]\\
			&+\cdots+\frac{3}{2}\sin\alpha\cos\alpha\left[\Psi(\alpha,l^*-2)+\Psi(\alpha,n-1)\right]\\
			&+\frac{3}{2}\sin\alpha\cos\alpha\left[\Psi(\alpha,l^*-1)+\Psi(\alpha,n)\right]+\frac{3}{2}\sin\alpha\cos\alpha\Psi(\alpha,l^*)\\
			=&\frac{3}{2}\sin\alpha\cos\alpha\sum_{l=0}^{2l^*-n-1}\Psi(\alpha,l)+\frac{3}{2}\sin\alpha\cos\alpha\Psi(\alpha,l^*)\\
			&+\frac{3}{2}\sin\alpha\cos\alpha\sum_{l=2l^*-n}^{l^*-1}\left[\Psi(\alpha,l)+\Psi(\alpha,l+n-l^*+1)\right].
		\end{aligned}
	\end{equation*}
	We will show $\Phi_n(\alpha)<0$, $\alpha\in(0,\frac{\pi}{6})$, $n\geq4$, through the following two steps.	
	
	\emph{Step 1}. When $l\in[0,2l^*-n-1]$, we have $\binom{2n+1}{2l+1}-\frac{1}{3}\binom{2n+1}{2l}>0$. Using the fact that $\sqrt{3}\sin\alpha<\cos\alpha$ for $\alpha\in(0,\frac{\pi}{6})$, it follows that
	\begin{equation*}
		\Psi(\alpha,l)<\left[\binom{2n+1}{2l+1}-\frac{1}{3}\binom{2n+1}{2l}\right]\left[\left(\frac{1}{2}\cos\alpha\right)^{2l}\left(\frac{1}{2}\cos\alpha\right)^{2n-2l}-\left(\frac{1}{2}\cos\alpha\right)^{2n}\right]=0.
	\end{equation*}
	Therefore we get $\frac{3}{2}\sin\alpha\cos\alpha\sum_{l=0}^{2l^*-n-1}\Psi(\alpha,l)<0$, $\alpha\in(0,\frac{\pi}{6})$. Similarly, when $l=l^*$, we have
	$\frac{3}{2}\sin\alpha\cos\alpha\Psi(\alpha,l^*)\leq0$ for $\alpha\in(0,\frac{\pi}{6})$.
	
	\emph{Step 2}. When $l\in[2l^*-n,l^*-1]$, we will prove that $\Psi(\alpha,l)+\Psi(\alpha,l+n-l^*+1)<0$, $\alpha\in(0,\frac{\pi}{6})$. For this purpose, we let
	\begin{equation*}
		\frac{\Psi(\alpha,l)}{\Psi(\alpha,l+n-l^*+1)}:=\Psi_1(\alpha,l;l^*)\Psi_2(l;l^*),
	\end{equation*}
	where
	\begin{equation}
		\label{eq2.2}
		\begin{aligned}
			\Psi_1(\alpha,l;l^*)&=
			\frac{\left(\frac{\sqrt{3}}{2}\sin\alpha\right)^{2l}\left(\frac{1}{2}\cos\alpha\right)^{2n-2l}-\left(\frac{1}{2}\cos\alpha\right)^{2n}}
			{\left(\frac{\sqrt{3}}{2}\sin\alpha\right)^{2l+2n-2l^*+2}\left(\frac{1}{2}\cos\alpha\right)^{2l^*-2l-2}-\left(\frac{1}{2}\cos\alpha\right)^{2n}}\\
			&=\frac{\left(\frac{1}{2}\cos\alpha\right)^{2n-2l^*+2}\left(\frac{\sqrt{3}}{2}\sin\alpha\right)^{2l}\left(\frac{1}{2}\cos\alpha\right)^{2l^*-2l-2}-\left(\frac{1}{2}\cos\alpha\right)^{2n}}
			{\left(\frac{\sqrt{3}}{2}\sin\alpha\right)^{2n-2l^*+2}\left(\frac{\sqrt{3}}{2}\sin\alpha\right)^{2l}\left(\frac{1}{2}\cos\alpha\right)^{2l^*-2l-2}-\left(\frac{1}{2}\cos\alpha\right)^{2n}}\\
			&>1,  
		\end{aligned}		
	\end{equation}
	and
	\begin{equation*}
		\begin{aligned}
			\Psi_2(l;l^*)=&\frac{\binom{2n+1}{2l+1}-\frac{1}{3}\binom{2n+1}{2l}}{\binom{2n+1}{2l+2n-2l^*+3}-\frac{1}{3}\binom{2n+1}{2l+2n-2l^*+2}}\\
			=&\frac{6n-8l+2}{-2n+8l^*-8l-6}\frac{(2n-2l^*+2l+3)!(2l^*-2l-1)!}{(2n-2l+1)!(2l+1)!}\\
			=&\frac{3n-4l+1}{-n+4l^*-4l-3}\frac{2l+2}{2l^*-2l}\frac{\prod_{k=0}^{2n-2l^*}(2n-2l^*+2l+3-k)}{\prod_{k=0}^{2n-2l^*}(2n-2l+1-k)}.
		\end{aligned}
	\end{equation*}	
	Since $\Psi_1(\alpha,l;l^*)>1$,  it remains to prove $\Psi_2(l;l^*)<-1$.
	
	By $l^*=\left[\frac{3n+1}{4}\right]$, we have $l^*=\left[\frac{3n-3}{4}\right]+1>\frac{3n-3}{4}$, that is, $n<\frac{4l^*+3}{3}$. Furthermore, when $n\geq4$, that is, $l^*\geq3$, it follows from $l\in[2l^*-n,l^*-1]$ that
	\begin{equation*}
		l\geq2l^*-n>2l^*-\frac{4l^*+3}{3}=\frac{2l^*-3}{3}\geq\frac{l^*-1}{2}.
	\end{equation*}
	This implies $l>l^*-l-1$. Moreover, it is easy to see that $\{\frac{2n-2l^*+2l+3-k}{2n-2l+1-k}\}_{k=0}^{2n-2l^*}$ is a positive increasing sequence. Hence, we obatin
	\begin{equation}
		\label{eq2.3}
		\begin{aligned}
			\frac{\prod_{k=0}^{2n-2l^*}(2n-2l^*+2l+3-k)}{\prod_{k=0}^{2n-2l^*}(2n-2l+1-k)}&=\prod_{k=0}^{2n-2l^*}\frac{2n-2l^*+2l+3-k}{2n-2l+1-k}\\
			&\geq\left(\frac{2n-2l^*+2l+3}{2n-2l+1}\right)^{2n-2l^*}\\
			&=\left(\frac{2n-2(l^*-l-1)+1}{2n-2l+1}\right)^{2n-2l^*}\\
			&>1.
		\end{aligned}	
	\end{equation}
	By $l\in[2l^*-n, l^*-1]$ and $l^*>\frac{3n-3}{4}$, it follows that
	\begin{equation*}
		l^*-l\leq n-l^*<n-\frac{3n-3}{4}=\frac{n+3}{4},
	\end{equation*}
	which, combining with $l\leq l^*-1$, implies $1\leq l^*-l<\frac{n+3}{4}$, and
	\begin{equation*}
		1-n\leq-n+4l^*-4l-3<0.
	\end{equation*}
	From $\frac{3n-3}{4}<l^*\leq\frac{3n+1}{4}$, it follows that $-3<4l^*-3n\leq1$. Thus we can conclude that
	\begin{equation}
		\label{eq2.4}
		\begin{aligned}
			\frac{3n-4l+1}{-n+4l^*-4l-3}\frac{2l+2}{2l^*-2l}&=-1-\frac{(8l^*-4n)l-(4l^*-n-3)l^*-3n-1}{4l^2-(8l^*-n-3)l+(4l^*-n-3)l^*}\\
			&<-1-\frac{(4l^*-3n+3)l^*-3n-1}{4l^2-(8l^*-n-3)l+(4l^*-n-3)l^*}\\
			&\leq-1-\frac{4l^*-3n-1}{4l^2-(8l^*-n-3)l+(4l^*-n-3)l^*}\\
			&\leq-1.
		\end{aligned}
	\end{equation}
	By \eqref{eq2.3} and \eqref{eq2.4}, we obtain $\Psi_2(l;l^*)<-1$, which, together with \eqref{eq2.2}, implies
	\begin{equation*}
		\Psi(\alpha,l)+\Psi(\alpha,l+n-l^*+1)<0, \ \alpha\in(0,\frac{\pi}{6}), \  l\in[2l^*-n, l^*-1].
	\end{equation*} 
	Obviously, $\Phi_n(0)=\Phi_n(\frac{\pi}{6})=0$ for  all $n\geq4$. This completes the proof.		
\end{proof}

Now we present a detailed  \textbf{proof of Theorem~\ref{th1.1}}.
\begin{proof}
Theorem \ref{th1.2} states that the constant $c^*=\inf_{\lambda>0}\frac{g(\lambda)}{\lambda}$ defined in \eqref{eq1.14} is the minimal wave speed of traveling waves of \eqref{eq1.17}. Recall that $(\kappa,\sigma)=(\cos\alpha,\sin\alpha)$, $\alpha\in[0,2\pi)$. To investigate the influence of  $\alpha$ on $c^*$, we denote $g(\lambda):=g(\lambda,\alpha)$ to highlight such relationship, and let 
\begin{equation}
	\label{eq2.5}	
	G(c,\lambda,\alpha):=c\lambda-g(\lambda,\alpha),  \ \lambda>0,
\end{equation}
where $c$ is a parameter. For the minimal wave speed, we need to find $\lambda^*>0$ such that $c^*=\frac{g(\lambda^*,\alpha)}{\lambda^*}$, which implies that $(c^*,\lambda^*)$ is a positive solution of the following system
\begin{equation*}
	\label{eq2.6}
	\left\{\begin{aligned}
		&G(c,\lambda,\alpha)=0,&&\lambda>0, \ \alpha\in[0,2\pi),\\
		&\frac{\partial{G(c,\lambda,\alpha)}}{\partial{\lambda}}=0, &&\lambda>0,  \ \alpha\in[0,2\pi).
	\end{aligned}\right.
\end{equation*}
Set
\begin{equation*}
	J:=\left(\begin{array}{cc}
		\frac{\partial{G(c,\lambda,\alpha)}}{\partial{c}}&\frac{\partial{G(c,\lambda,\alpha)}}{\partial{\lambda}}\\
		\frac{\partial^2{G(c,\lambda,\alpha)}}{\partial{c}\partial{\lambda}}&\frac{\partial^2{G(c,\lambda,\alpha)}}{\partial{\lambda^2}}
	\end{array}\right)\Bigg|_{(c^*,\lambda^*)}.
\end{equation*}
It follows from \eqref{eq1.14} that $\frac{\partial^2 g(\lambda,\alpha)}{\partial \lambda^2}>0$ for $\lambda>0$, $\alpha\in[0,2\pi)$, which implies $\frac{\partial^2 G(c,\lambda,\alpha)}{\partial\lambda^2}\big|_{(c^*,\lambda^*)}=-\frac{\partial^2 g(\lambda,\alpha)}{\partial \lambda^2}\big|_{(c^*,\lambda^*)}<0$. Combining $\frac{\partial G(c,\lambda,\alpha)}{\partial c}\big|_{(c^*,\lambda^*)}=\lambda^*>0$ with $\frac{\partial G(c,\lambda,\alpha)}{\partial\lambda}\big|_{(c^*,\lambda^*)}=0$, we obtain $|J|\neq0$. By the implicit function theorem, it follows that there exist two continuous functions $\lambda^*=\lambda^*(\alpha)$ and $c^*=c^*(\alpha)$ with respect to $\alpha\in[0,2\pi)$. Moreover,
\begin{equation*}		
	\frac{\mathrm{d} c^*(\alpha)}{\mathrm{d} \alpha}=-\frac{\frac{\partial G(c,\lambda,\alpha)}{\partial\alpha}}{\frac{\partial G(c,\lambda,\alpha)}{\partial c}} \Bigg|_{(c^*,\lambda^*)}.
\end{equation*}
Together with \eqref{eq2.5}, we have
\begin{equation*}
	\mathrm{sign}\left(\frac{\mathrm{d} c^*(\alpha)}{\mathrm{d} \alpha}\right)=\mathrm{sign}\left(-\frac{\partial G(c,\lambda,\alpha)}{\partial\alpha}\bigg|_{(c^*,\lambda^*)}\right)=\mathrm{sign}\left(\frac{\partial g(\lambda,\alpha)}{\partial\alpha}\bigg|_{(c^*,\lambda^*)}\right).
\end{equation*}
By substituting  $\kappa$, $\sigma$ with $\cos\alpha$, $\sin\alpha$, respectively, in the expression $g(\lambda):=g(\lambda,\alpha)$, we derive
\begin{equation*}	
	\begin{aligned}
		&\frac{\partial g(\lambda,\alpha)}{\partial\alpha}\bigg|_{(c^*,\lambda^*)}\\	
		=&\frac{\lambda^*(\alpha)}{6}\Bigg[\sin\alpha\left(\mathrm{e}^{-\lambda^*(\alpha)\cos\alpha}-\mathrm{e}^{\lambda^*(\alpha)\cos\alpha}\right)\\
		&+\left(\frac{\sqrt{3}}{2}\cos\alpha+\frac{1}{2}\sin\alpha\right)\left(\mathrm{e}^{\lambda^*(\alpha)(\frac{\sqrt{3}}{2}\sin\alpha-\frac{1}{2}\cos\alpha)}-\mathrm{e}^{-\lambda^*(\alpha)(\frac{\sqrt{3}}{2}\sin\alpha-\frac{1}{2}\cos\alpha)}\right)\\
		&+\left(\frac{\sqrt{3}}{2}\cos\alpha-\frac{1}{2}\sin\alpha\right)\left(\mathrm{e}^{\lambda^*(\alpha)(\frac{\sqrt{3}}{2}\sin\alpha+\frac{1}{2}\cos\alpha)}-\mathrm{e}^{-\lambda^*(\alpha)(\frac{\sqrt{3}}{2}\sin\alpha+\frac{1}{2}\cos\alpha)}\right)\Bigg].			
	\end{aligned}	
\end{equation*}
For $\alpha\in[0,2\pi)$, by applying $\mathrm{e}^{z}=\sum_{n=0}^{\infty}\frac{z^n}{n!}$, $z\in \mathbb{R}$, to the above formula, we get
\begin{equation}
	\label{eq2.10}	
	\begin{aligned}
		& \frac{\partial g(\lambda,\alpha)}{\partial\alpha}\bigg|_{(c^*,\lambda^*)} \\
		=&-\frac{\lambda^*(\alpha)}{3}\sin\alpha\sum_{n=0}^{\infty}\frac{(\lambda^*(\alpha)\cos\alpha)^{2n+1}}{(2n+1)!}\\
		&+\frac{\lambda^*(\alpha)}{3}\left(\frac{\sqrt{3}}{2}\cos\alpha+\frac{1}{2}\sin\alpha\right)\sum_{n=0}^{\infty}\frac{(\lambda^*(\alpha)(\frac{\sqrt{3}}{2}\sin\alpha-\frac{1}{2}\cos\alpha))^{2n+1}}{(2n+1)!}\\
		&+\frac{\lambda^*(\alpha)}{3}\left(\frac{\sqrt{3}}{2}\cos\alpha-\frac{1}{2}\sin\alpha\right)\sum_{n=0}^{\infty}\frac{(\lambda^*(\alpha)(\frac{\sqrt{3}}{2}\sin\alpha+\frac{1}{2}\cos\alpha))^{2n+1}}{(2n+1)!}\\
		=&\sum_{n=0}^{\infty}\frac{(\lambda^*(\alpha))^{2n+2}}{3(2n+1)!}\Bigg[\left(\frac{\sqrt{3}}{2}\cos\alpha+\frac{1}{2}\sin\alpha\right)\left(\frac{\sqrt{3}}{2}\sin\alpha-\frac{1}{2}\cos\alpha\right)^{2n+1}\\
		&+\left(\frac{\sqrt{3}}{2}\cos\alpha-\frac{1}{2}\sin\alpha\right)\left(\frac{\sqrt{3}}{2}\sin\alpha+\frac{1}{2}\cos\alpha\right)^{2n+1}-\sin\alpha(\cos\alpha)^{2n+1}\Bigg].
	\end{aligned}	
\end{equation}	
Furthermore, we can calculate that
\begin{equation*}
	\frac{g(\lambda,\alpha)}{\lambda}=\frac{g(\lambda,\alpha+\frac{\pi}{3})}{\lambda}, \  \frac{g(\lambda,\frac{\pi}{6}-\alpha)}{\lambda}=\frac{g(\lambda,\frac{\pi}{6}+\alpha)}{\lambda} \ \text{for all}\ \lambda>0, 
\end{equation*}
which implies that $c^*(\alpha)$ is $\frac{\pi}{3}$-periodic with $\alpha\in[0,2\pi)$ and symmetric about $\alpha=\frac{\pi}{6}$ in $[0,\frac{\pi}{3}]$. It remains to prove the monotonicity of $\frac{\partial g(\lambda,\alpha)}{\partial\alpha}\big|_{(c^*,\lambda^*)}$ in $[0,\frac{\pi}{6}]$. From \eqref{eq2.10} and $\Phi_n(\alpha)$ in Lemma~\ref{le2.1}, we see that
\begin{equation*}
	\frac{\partial g(\lambda,\alpha)}{\partial\alpha}\bigg|_{(c^*,\lambda^*)}=\sum_{n=0}^{\infty}\frac{(\lambda^*(\alpha))^{2n+2}}{3(2n+1)!}\Phi_n(\alpha), \ \alpha\in[0,\frac{\pi}{6}], \ n\in\mathbb{N}.
\end{equation*}	
From Lemma~\ref{le2.1}, we can conclude that $\frac{\mathrm{d}c^*(\alpha)}{\mathrm{d}\alpha}<0$ in $\alpha\in(0,\frac{\pi}{6})$ and $\frac{\mathrm{d}c^*(\alpha)}{\mathrm{d}\alpha}\big|_{\alpha=0,\frac{\pi}{6}}=0$. The proof is completed.		
\end{proof}

\section{Numerical simulations} \label{sec:Nume}
Theorem~\ref{th1.2} establishes the existence of a minimal speed  $c^*$ for traveling wave solutions, which depends on the angle in a highly complicated manner, as described by the expression~\eqref{eq1.14}. To validate our theoretical results and gain further insight into the minimal wave speed, we present numerical simulations in this section. Additionally, we numerically compute the spreading speed of~\eqref{eq1.5} by solving an initial value problem and investigate its relationship with the minimal wave speed.

\subsection{The minimal wave speed}
To illustrate the dependence of the minimal wave speed on an angle, here we present a numerical graph that shows this relationship, as given in Theorem~\ref{th1.1}. Specifically, we set $f'(0)=10$ and choose 2000 different values of $\lambda$ in \eqref{eq1.14}, that is, $\lambda_n=0.01n$, $n=1, 2, \dots, 2000$. This setup allows us to effectively depict the periodicity and monotonicity of $c^*(\alpha)$ with respect to $\alpha$, as shown in Figure~\ref{fig2}. Furthermore, the graph reveals that $c^*(\alpha)$ attains its maximum values at angles 
$\alpha=0,\frac{\pi}{3}$, $\frac{2\pi}{3}$, $\pi,\frac{4\pi}{3}$, $\frac{5\pi}{3}$, and its minimum values at angles $\alpha=\frac{\pi}{6}$, $\frac{\pi}{2}$, $\frac{5\pi}{6}$, $\frac{7\pi}{6}$, $\frac{3\pi}{2},\frac{11\pi}{6}$. This periodic behavior underscores the complex relationship between the wave speed and a directional angle, highlighting the complex dynamics captured by our model.

\begin{figure}[htbp]
	\centering
	%\subfigure []
	{\includegraphics[width=2.7in,height=2.7in,clip]{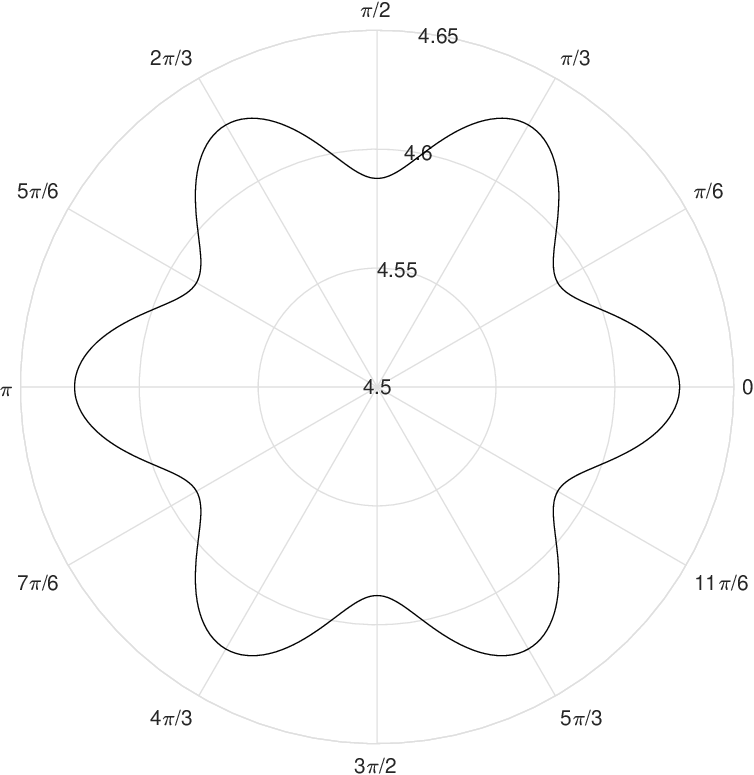} }
	\caption{The minimal wave speed  $c^*$ with respect to $\alpha$ in the polar coordinate system. The coordinates on the radius represent the minimal wave speed with range $[4.5,4.65]$.}\label{fig2}	
\end{figure}

The expression $\Phi_n(\alpha)$ with $[0,\frac{\pi}{6}]$ and $n\in\mathbb{N}$, defined in Lemma \ref{le2.1}, plays a key role in proving the monotonicity of $c^*$ with respect to $\alpha$ in Theorem~\ref{th1.1}. We can easily verify that $\Phi_n(\alpha)$ is $\frac{\pi}{3}$-periodic with $\alpha\in[0,2\pi)$ and symmetric about $\alpha=\frac{\pi}{6}$ in $[0,\frac{\pi}{3}]$ for $n\in\mathbb{N}$ after extending $[0,\frac{\pi}{6}]$ to $[0,2\pi)$, which is supported by the plot (Figure \ref{fig3}) of $\Phi_n(\alpha)$ on $[0,2\pi)$ for a number of different values of $n$. Moreover, Figure \ref{fig3} illustrates $\Phi_n(\alpha)\leq0$ on $[0,\frac{\pi}{6}]$ for $n=1, 2, 4, 20, 50, 100$ in Lemma \ref{le2.1}.

\begin{figure}[htbp]
	\centering
	%\subfigure []
	{\includegraphics[width=4in,height=3in,clip]{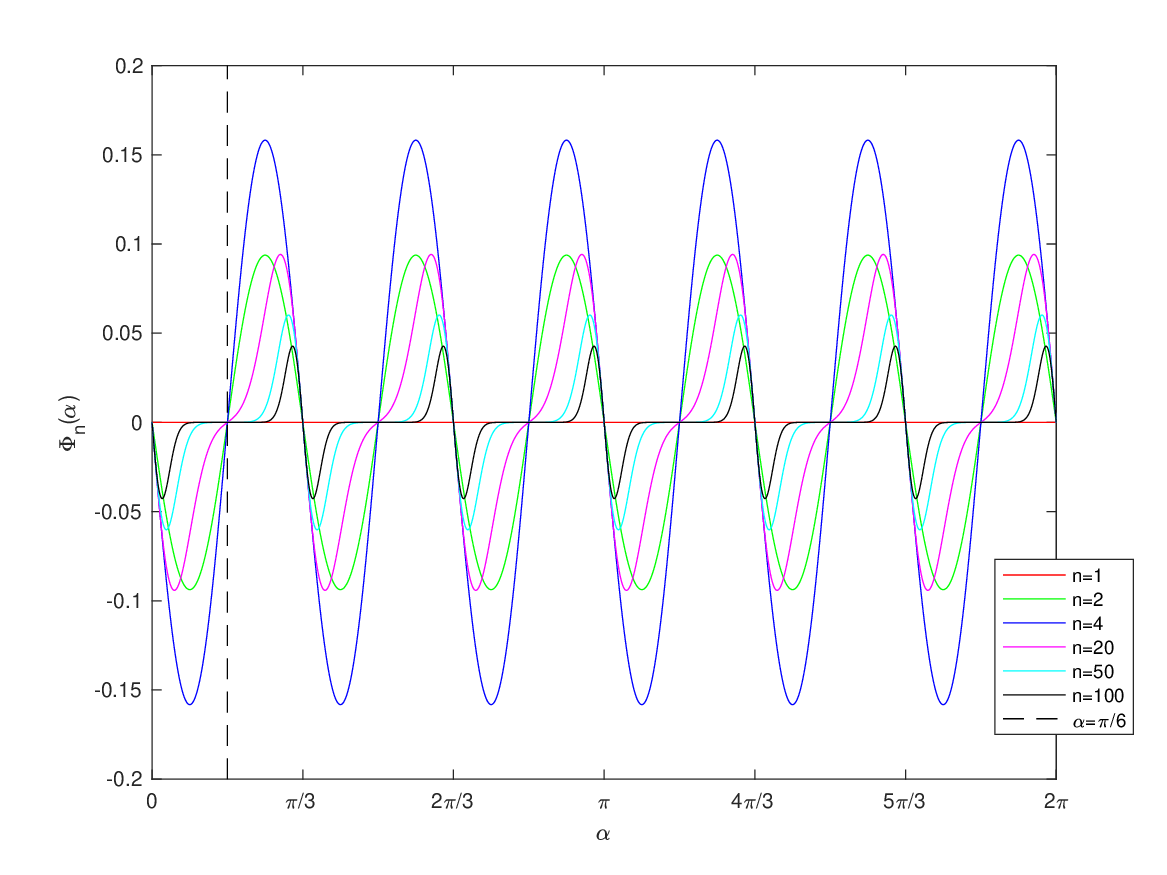} }
	\caption{The expression $\Phi_n(\alpha)$ with respect to $\alpha$ for different $n$.}\label{fig3}	
\end{figure}

\subsection{The spreading speed}

The spreading speed can be numerically computed using solutions of~\eqref{eq1.5} with the initial condition $u(x,y,0) := u_0(x,y) \in [0,1]$. However, the operator $\Delta_h$ in~\eqref{eq1.5} contains terms with irrational coordinates, which presents challenges for numerical discretization. To address this, we first find an invertible matrix transformation to map the hexagonal lattice onto a square lattice with a diagonal orientation, as illustrated in Figure~\ref{fig11}. The initial value problem is then solved numerically on the transformed square lattice.

\begin{figure}[htbp]
	\centering
	\begin{tikzpicture}[>=stealth,thick]
		\node (s00) at (0,0) [circle, draw=black]{};
		\node (s10) at (2,0) [circle, draw, fill=blue!50]{};
		\node (s20) at (4,0) [circle, draw, fill=red!50]{};
		\draw [dashed,] (s20)--(4,-1); \draw [dashed,] (s20)--(5,1); \draw [dashed,] (s20)--(5,0);
		\node (s-10) at (-2,0) [circle, draw, fill=blue!50]{};
		\node (s-20) at (-4,0) [circle, draw, fill=red!50]{};
		\node (s-11) at (-2,2) [circle, draw, fill=red!50]{};
		\draw [dashed,] (s-11)--(-2,3); \draw [dashed,] (s-11)--(-3,2);
		\node (s01) at (0,2) [circle, draw, fill=blue!50]{};
		\node (s11) at (2,2) [circle, draw, fill=blue!50]{};
		\node (s21) at (4,2) [circle, draw, fill=red!50]{};
		\draw [dashed,] (s21)--(5,3); \draw [dashed,] (s21)--(5,2);
		\node (s02) at (0,4) [circle, draw, fill=red!50]{};
		\node (s12) at (2,4) [circle, draw, fill=red!50]{};
		\node (s22) at (4,4) [circle, draw, fill=red!50]{};
		\node (s-2-1) at (-4,-2) [circle, draw, fill=red!50]{};
		\node (s-1-1) at (-2,-2) [circle, draw, fill=blue!50]{};
		\node (s0-1) at (0,-2) [circle, draw, fill=blue!50]{};
		\node (s1-1) at (2,-2) [circle, draw, fill=red!50]{};
		\draw [dashed,] (s1-1)--(2,-3); \draw [dashed,] (s1-1)--(3,-2);
		\node (s-2-2) at (-4,-4) [circle, draw, fill=red!50]{};
		\node (s-1-2) at (-2,-4) [circle, draw, fill=red!50]{};
		\node (s0-2) at (0,-4) [circle, draw, fill=red!50]{};
		\draw [-] (s02)--(s12)node[]{}; 
		\draw [dashed,] (s02)--(0,5); \draw [dashed,] (s02)--(1,5); \draw [dashed,] (s02)--(-1,4);
		\draw [-] (s12)--(s22)node[]{};
		\draw [dashed,] (s12)--(2,5); \draw [dashed,] (s12)--(3,5); 
		\draw [-] (s-11)--(s01)node[]{};
		\draw [-] (s01)--(s11)node[]{};
		\draw [-] (s11)--(s21)node[]{};
		\draw [-] (s-20)--(s-10)node[]{};
		\draw [-] (s-10)--(s00)node[]{};
		\draw [-] (s00)--(s10)node[]{};
		\draw [-] (s10)--(s20)node[]{};
		\draw [-] (s-2-1)--(s-1-1)node[]{};
		\draw [-] (s-1-1)--(s0-1)node[]{};
		\draw [-] (s0-1)--(s1-1)node[]{};
		\draw [-] (s-2-2)--(s-1-2)node[]{};
		\draw [-] (s-1-2)--(s0-2)node[]{};
		\draw [-] (s-20)--(s-2-1)node[]{};
		\draw [dashed,] (s-20)--(-4,1); \draw [dashed,] (s-20)--(-5,-1); \draw [dashed,] (s-20)--(-5,0);
		\draw [-] (s-2-1)--(s-2-2)node[]{};
		\draw [-] (s-11)--(s-10)node[]{};
		\draw [-] (s-10)--(s-1-1)node[]{};
		\draw [-] (s-1-1)--(s-1-2)node[]{};
		\draw [-] (s02)--(s01)node[]{};
		\draw [-] (s01)--(s00)node[]{};
		\draw [-] (s00)--(s0-1)node[]{};
		\draw [-] (s0-1)--(s0-2)node[]{};
		\draw [-] (s12)--(s11)node[]{};
		\draw [-] (s11)--(s10)node[]{};
		\draw [-] (s10)--(s1-1)node[]{};
		\draw [-] (s22)--(s21)node[]{};
		\draw [dashed,] (s22)--(4,5); \draw [dashed,] (s22)--(5,5); \draw [dashed,] (s22)--(5,4);
		\draw [-] (s21)--(s20)node[]{};
		\draw [-] (s-2-2)--(s-1-1)node[]{};
		\draw [dashed,] (s-2-2)--(-4,-5); \draw [dashed,] (s-2-2)--(-5,-5); \draw [dashed,] (s-2-2)--(-5,-4);
		\draw [-] (s-1-1)--(s00)node[]{};
		\draw [-] (s00)--(s11)node[]{};
		\draw [-] (s11)--(s22)node[]{};   
		\draw [-] (s-2-1)--(s-10)node[]{};
		\draw [dashed,] (s-2-1)--(-5,-3); \draw [dashed,] (s-2-1)--(-5,-2);
		\draw [-] (s-10)--(s01)node[]{};
		\draw [-] (s01)--(s12)node[]{};
		\draw [-] (s-20)--(s-11)node[]{}; 
		\draw [-] (s-11)--(s02)node[]{};
		\draw [-] (s-1-2)--(s0-1)node[]{};
		\draw [dashed,] (s-1-2)--(-2,-5); \draw [dashed,] (s-1-2)--(-3,-5); 
		\draw [-] (s0-1)--(s10)node[]{};
		\draw [-] (s10)--(s21)node[]{};
		\draw [-] (s0-2)--(s1-1)node[]{}; 
		\draw [dashed,] (s0-2)--(0,-5); \draw [dashed,] (s0-2)--(-1,-5); \draw [dashed,] (s0-2)--(1,-4);
		\draw [-] (s1-1)--(s20)node[]{};
		\node at (0.5,3.7) []{\footnotesize$(0,2)$};
		\node at (2.5,3.7) []{\footnotesize$(1,2)$};
		\node at (4.5,3.7) []{\footnotesize$(2,2)$};
		\node at (-1.4,1.7) []{\footnotesize$(-1,1)$};
		\node at (0.5,1.7) []{\footnotesize$(0,1)$};
		\node at (2.5,1.7) []{\footnotesize$(1,1)$};
		\node at (4.5,1.7) []{\footnotesize$(2,1)$};
		\node at (-3.4,-0.3) []{\footnotesize$(-2,0)$};
		\node at (-1.4,-0.3) []{\footnotesize$(-1,0)$};
		\node at (0.5,-0.3) []{\footnotesize$(0,0)$};
		\node at (2.5,-0.3) []{\footnotesize$(1,0)$};
		\node at (4.5,-0.3) []{\footnotesize$(2,0)$};
		\node at (-3.3,-2.3) []{\footnotesize$(-2,-1)$};
		\node at (-1.3,-2.3) []{\footnotesize$(-1,-1)$};
		\node at (0.6,-2.3) []{\footnotesize$(0,-1)$};
		\node at (2.6,-2.3) []{\footnotesize$(1,-1)$};
		\node at (-3.3,-4.3) []{\footnotesize$(-2,-2)$};
		\node at (-1.3,-4.3) []{\footnotesize$(-1,-2)$};
		\node at (0.6,-4.3) []{\footnotesize$(0,-2)$};
		\draw [->] (-7,-4)--(-7,-2)node[left=0.05cm]{\footnotesize$j$};
		\draw [->] (-8,-3)--(-6,-3)node[right=0.05cm]{\footnotesize$i$};	
	\end{tikzpicture}
	\caption{A square lattice with a diagonal line described by the coordinates $i$-$j$.}\label{fig11}	
\end{figure}
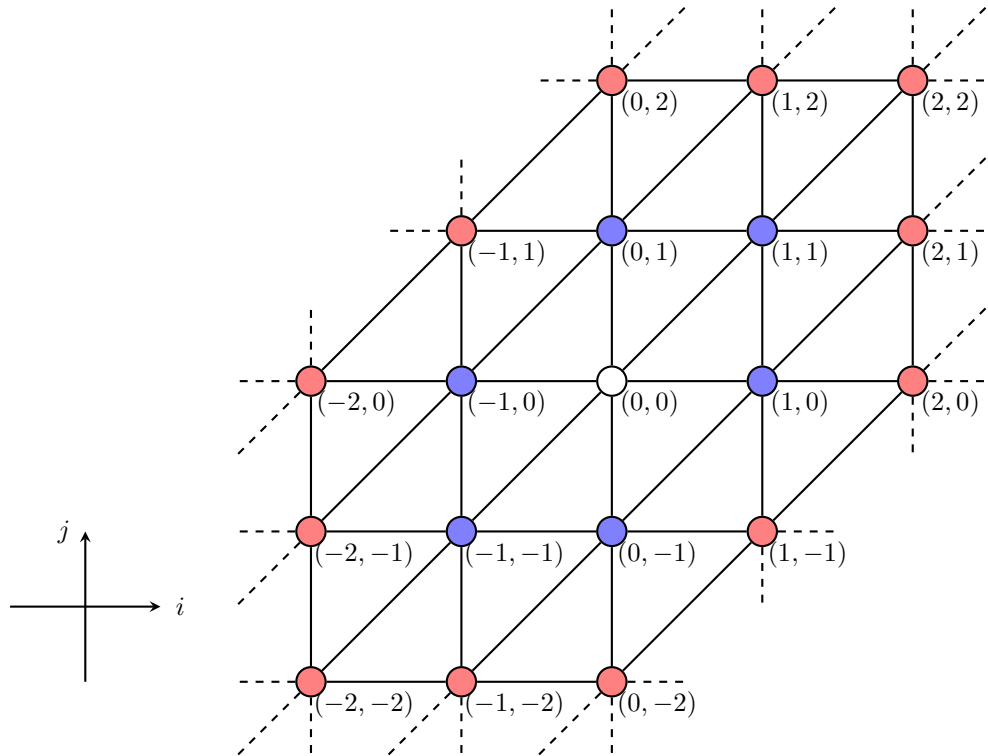

To perform  the transition between coordinate systems, we introduce the invertible matrix $A$, defined by
\begin{equation*}
	\begin{pmatrix}
		i\\
		j
	\end{pmatrix}
	= A
	\begin{pmatrix}
		x\\
		y
	\end{pmatrix}
	:= \begin{pmatrix}
		1 & \frac{\sqrt{3}}{3}\\
		0 & \frac{2\sqrt{3}}{3}
	\end{pmatrix}
	\begin{pmatrix}
		x\\
		y
	\end{pmatrix}, \quad (x, y) \in \mathbb{H},
\end{equation*}
where $(x, y)$ are the hexagonal coordinates and $(i, j)$ are the corresponding square lattice coordinates. A straightforward calculation yields $(i, j) = \left(x + \frac{\sqrt{3}}{3}y, \frac{2\sqrt{3}}{3}y\right) \in \mathbb{Z}^2$. Under this transformation, the operator $\Delta_h$ is converted to $\Delta_s$, which acts on functions $v: \mathbb{Z}^2 \times \mathbb{R} \rightarrow \mathbb{R}$ and is defined by
\begin{equation*}
	\Delta_s[v]_{i,j}(\cdot) := v_{i+1,j}(\cdot) + v_{i-1,j}(\cdot) + v_{i,j+1}(\cdot) + v_{i,j-1}(\cdot) + v_{i+1,j+1}(\cdot) + v_{i-1,j-1}(\cdot) - 6v_{i,j}(\cdot).
\end{equation*}
	
Next, we numerically solve the equation
\begin{equation*}
	v'(i, j, t) = \frac{1}{6} \Delta_s[v]_{i, j}(t) + f(v(i, j, t)), \quad (i, j) \in \mathbb{Z}^2, \ t > 0,
\end{equation*}
where $f(v) = av(1-v)$ with $a = 200$. The initial condition is $v(i, j, 0)= 1$ for $i,j=0$ and $v(i, j, 0)= 0$ elsewhere. The size of the domain is set to $\{(i,j)\in[-200, 200]\times[-200, 200]\}$. We record the times at which the solution first reaches $0.5$ on those nodes located on four rectangular rings, where the rings are centred around $(0,0)$ with side lengths $5$, $10$, $20$ and $80$, respectively. We choose all nodes on the shortest ring with side length $5$ and some nodes on other rings sparsely so that these nodes correspond one by one in all directions. By applying the inverse matrix $A$, we transform the coordinates of these nodes to the $x$-$y$ coordinate system. Note that the straight path remains straight after the transition. We then calculate the average speed between two nodes in one direction on the $n$-th ring and the $(n+1)$-th ring, denoted as $\bar{c}_n(\alpha)$, $n=1,2,3$. As shown in Figure~\ref{fig5}, the change of $\bar{c}_n(\alpha)$ suggests that there exists a spreading speed, which coincides with the minimal wave speed in all directions.

\begin{figure}[htbp]
	\centering
	{\includegraphics[width=3in,height=2.5in,clip]{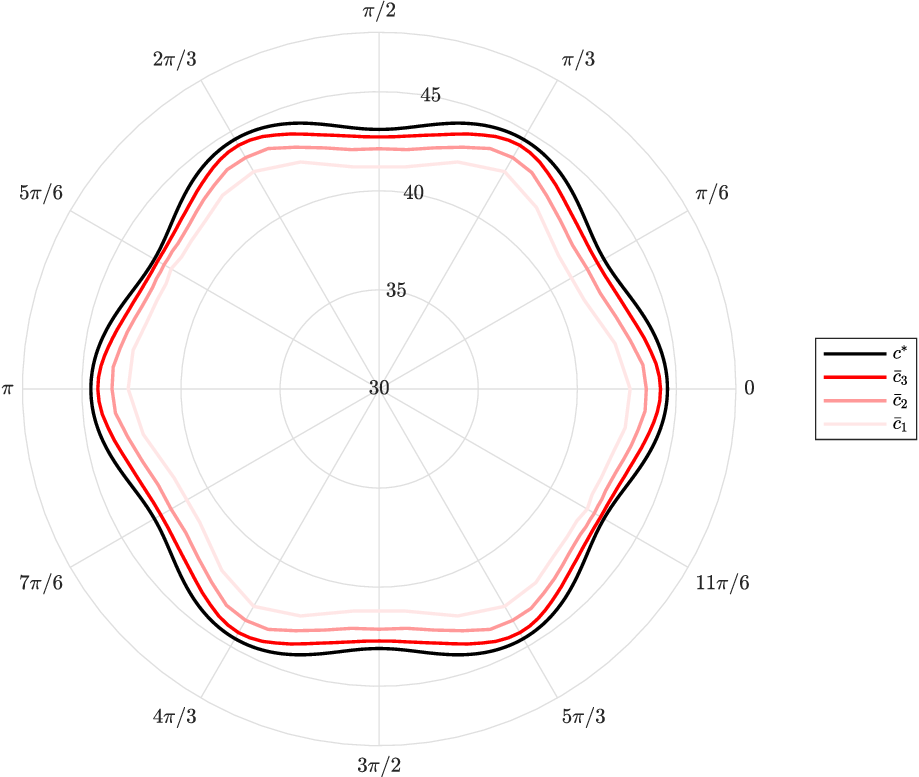} }
	\caption{Speeds $\bar{c}_1$, $\bar{c}_2$, $\bar{c}_3$ and $c^*$ with respect to $\alpha$ in the polar coordinate system. The coordinates on the radius represent the minimal wave speed or spreading speed with range $[30,48]$.}\label{fig5}	
\end{figure}

\section{Discussion} \label{sec:Dis}

Comparing with the existing studies for LDES on one-dimensional line grids or two-dimensional square lattices, this study broadens the scope by examining species dispersal within a honeycomb neighborhood structure, which would be more suitable in certain natural and habitat environments. Our model carefully describes individual movements within a hexagonal network, involving an accurate, while mathematically tractable, formulation of the neighboring nodes within a coordinate system. This approach enables us to explore how individuals propagate through such a complex structure, offering insights that are not readily apparent in simpler lattice models.

\subsection{Properties of the minimal wave speed on a square lattice}

It is interesting to note that the approach used to prove Theorem \ref{th1.1} can be readily adapted to analyze the periodicity and monotonicity of the minimal wave speed of \eqref{eq1.4} with respect to the angle on a two-dimensional square lattice. Let $\beta\in[0,2\pi)$ be an angle on the square lattice. Then, according to \cite{GW2}, we see that the minimal wave speed of \eqref{eq1.4} can be expressed by 
\begin{equation*}
	\begin{aligned}
		c_s^*:=\inf_{\nu>0}\frac{\mathrm{e}^{\nu\cos\beta}+\mathrm{e}^{-\nu\cos\beta}+\mathrm{e}^{\nu\sin\beta}+\mathrm{e}^{-\nu\sin\beta}-4+f'(0)}{\nu}:=\inf_{\nu>0}\frac{g_s(\nu,\beta)}{\nu}.
	\end{aligned}
\end{equation*}
From the same arguments  used to determine   $c^*(\alpha)$, we infer that there exists $\nu^*=\nu^*(\beta)$ and 
\begin{equation*}
	c_s^*(\beta)=\frac{g_s(\nu^*(\beta),\beta)}{\nu^*(\beta)}, \ \beta\in[0,2\pi).
\end{equation*}
Applying the method used to prove Theorem~\ref{th1.1}, we establish the monotonicity and periodicity of $c_s^*(\beta)$ with respect to $\beta$, as given in the next remark.
\begin{remark}
	\label{re4.1}
	The quantity $c_s^*(\beta)$ is $\frac{\pi}{2}$-periodic with $\beta\in[0,2\pi)$, monotonically decreasing in $\beta\in[0,\frac{\pi}{4}]$ and monotonically increasing in $\beta\in[\frac{\pi}{4},\frac{\pi}{2}]$. 
\end{remark}
\begin{proof}
Using the implicit function theorem, we can obtain
\begin{equation*}	
	\mathrm{sign}\left(\frac{\mathrm{d} c_s^*(\beta)}{\mathrm{d} \beta}\right)=\mathrm{sign}\left(\frac{\partial g_s(\nu,\beta)}{\partial\beta}\bigg|_{(c_s^*,\nu^*)}\right).
\end{equation*}
By $\mathrm{e}^{z}=\sum_{n=0}^{\infty}\frac{z^n}{n!}$, $z\in \mathbb{R}$, it follows that 
\begin{equation*}
	\begin{aligned}
		\frac{\partial g_s(\nu,\beta)}{\partial\beta}\bigg|_{(c_s^*,\nu^*)}=&\nu^*(\beta)\left[\sin\beta\left(\mathrm{e}^{-\nu^*(\beta)\cos\beta}-\mathrm{e}^{\nu^*(\beta)\cos\beta}\right)+\cos\beta\left(\mathrm{e}^{\nu^*(\beta)\sin\beta}-\mathrm{e}^{-\nu^*(\beta)\sin\beta}\right)\right]\\
		=&\sum_{n=0}^{\infty}\frac{2(\nu^*(\beta))^{2n+2}}{(2n+1)!}\sin\beta\cos\beta\left[(\sin\beta)^{2n}-(\cos\beta)^{2n}\right].		
	\end{aligned}
\end{equation*}
Therefore we get
\begin{equation*}
	\frac{\mathrm{d} c_s^*(\beta)}{\mathrm{d} \beta}\left\{ \begin{aligned}
		&<0, &&\beta\in(0,\frac{\pi}{4}),\\
		&>0, &&\beta\in(\frac{\pi}{4},\frac{\pi}{2}),\\
		&=0, &&\beta=0, \frac{\pi}{4}, \frac{\pi}{2}.
	\end{aligned}\right.
\end{equation*}
Furthermore, we can conclude that
\begin{equation*}
	\frac{g_s(\nu,\beta)}{\nu}=\frac{g_s(\nu,\beta+\frac{\pi}{2})}{\nu}, \   \frac{g_s(\nu,\frac{\pi}{4}-\beta)}{\nu}=\frac{g_s(\nu,\frac{\pi}{4}+\beta)}{\nu}, \ \beta\in[0,2\pi), \ \nu>0.
\end{equation*}
This completes the proof.
\end{proof}

\subsection{Outlook}

There remain many interesting and important questions regarding propagation dynamics on hexagonal lattices that are worth further investigation. For example, the coincidence of the minimal wave speed and spreading speed observed numerically needs to be proved. The stability of traveling waves in different senses is an interesting topic. Additionally, analyzing, or even rigorously defining, the generalized traveling waves along irrational directions remains a challenging problem. Studying the angular dependence of bistable equations is a natural next step. We anticipate that the bistable case will exhibit richer dynamics, including the pinning phenomenon and a more complex angular dependence of the wave speed. Furthermore, extending our model to study spreading dynamics on lattices with other complex dispersal mechanisms represents a promising direction for future research. Such mechanisms may include non-local dispersal, modeled by multi-layer connections, or fat-tailed dispersal kernels, which may result in accelerating spreading fronts. Finally, we also hope to find applications of this lattice framework to real-world phenomena, such as population dynamics and crystal growth.

%\section*{Acknowledgments}
%Jian Fang and Jian Wang are supported in part by NSF of China (No. 12171119). Yifei Li is supported in part by NSF of China (No. 12301624). Yijun Lou is partially supported by The Hong Kong Research Grants Council (No. 15304821).

%{\color{red}The authors would like to thank the handling editor and two anonymous reviewers for their insightful comments and suggestions, which have greatly improved our manuscript.}

\end{document}